\documentclass[12pt,draftcls,onecolumn]{IEEEtran}

\usepackage[utf8]{inputenc}
\usepackage{amsfonts,amsmath,amssymb,amsthm,bbm}
\usepackage{graphicx}
\usepackage{algorithm}
\usepackage{algorithmic}
\usepackage{caption}
\usepackage{tikz}
\usepackage{subfig}
\usepackage{xspace}
\newtheorem{assum}{Assumption}
\newtheorem{defn}{Definition}
\newtheorem{prop}{Proposition}

\newtheorem{lem}{Lemma}
\newtheorem{thm}{Theorem}
\theoremstyle{remark}
\newtheorem{rem}{Remark}
\graphicspath{ {./figures/} }

\newcommand{\tarc}{\mbox{\large$\frown$}}
\newcommand{\arc}[1]{\stackrel{\tarc}{#1}}
\newcommand{\tr}{\mathop{\mathrm{tr}}}

\newcommand{\Pos}{\mathrm{Pos}}
\newcommand{\conv}{\mathop{\mathrm{conv}}}
\newcommand{\diam}{\mathop{\mathrm{diam}}}

\newcommand{\arithgossip}{Arithmetic Gossip\xspace}
\newcommand{\midgossip}{Midpoint Gossip\xspace}
\newenvironment{varalgorithm}[1]
{\algorithm\renewcommand{\thealgorithm}{#1}}
{\endalgorithm}

\begin{document}

\title{Random Pairwise Gossip on $CAT(\kappa)$ Metric Spaces}
\author{\IEEEauthorblockN{Anass Bellachehab}
  \IEEEauthorblockA{
    Institut Mines-Télécom\\
    Télécom SudParis\\
    CNRS UMR 5157 SAMOVAR\\
    9 rue Charles Fourier, 91000 Evry\\
    France}
\and

\IEEEauthorblockN{J\'er\'emie Jakubowicz}
  \IEEEauthorblockA{
    Institut Mines-Télécom\\
    Télécom SudParis\\
    CNRS UMR 5157 SAMOVAR\\
    9 rue Charles Fourier, 91000 Evry\\
    France}
}
\thanks{This work was supported in part by a ``Futur et Rupture'' grant from the Institut Mines-Télécom, and by a ``Chaire Mixte'' from the Centre National de la Recherche Scientifique (CNRS) and Institut Mines-Télécom.}
\maketitle

\begin{abstract}
In the context of sensor networks, gossip algorithms are a popular, well esthablished technique for achieving consensus when sensor data is encoded in linear spaces. Gossip algorithms also have several extensions to non linear data spaces. Most of these extensions deal with Riemannian manifolds and use Riemannian gradient descent. This paper, instead, exhibits a very simple metric property that do not rely on any differential structure. This property strongly suggests that gossip algorithms could be studied on a broader family than Riemannian manifolds. And it turns out that, indeed, (local) convergence is guaranteed as soon as the data space is a mere $CAT(\kappa)$ metric space. We also study convergence speed in this setting and establish linear rates for $CAT(0)$ spaces, and local linear rates for $CAT(\kappa)$ spaces with $\kappa > 0$. Numerical simulations on several scenarii, with corresponding state spaces that are either Riemannian manifolds -- as in the problem of positive definite matrices consensus -- or bare metric spaces -- as in the problem of arms consensus -- validate the results. This shows that not only does our metric approach allows for a simpler and more general mathematical analysis but also paves the way for new kinds of applications that go beyond the Riemannian setting.
\smallskip
\end{abstract}

\section{Introduction}
The consensus problem is a fundamental problem in the theory of distributed systems. It appears in a variety of settings such as database management \cite{burrows:2006}, clock synchronization \cite{schenato:2007}, and signal estimation in wireless sensor networks \cite{schizas:2008}. In the context of sensor networks we require the agents to agree on some quantity (for example, deciding on an average temperature or a power level...); the sensors are also subjected to hardware and energy constraints which makes long range communications unreliable. Each sensor has only access to local information and can communicate with its nearest neighbors; there is no central fusion node. If the measurements belong to some vector space, \emph{e.g.} temperatures, speeds, or locations; \emph{Gossip} protocols (see, \emph{e.g.}, \cite{boyd:gossip}) are efficient candidates that converge with exponential speed towards a consensus state, assuming the network is connected.

However, there are several interesting cases where measurements cannot be added or scaled as vectors. Camera orientations are such an example: it does not make sense to add two orientations. There are several other examples of interest: subspaces, curves, angles which have no underlying vector space structure. Several approaches have been proposed in order to generalize the gossip algorithm to these nonlinear data spaces. In~\cite{bonnabel2011stochastic} consensus is seen as a problem of stochastic approximation in which a disagreement function is minimized; \cite{bonnabel2011stochastic} then proposes a gossip algorithm analogous to that of~\cite{boyd:gossip} in the case of Riemannian manifolds of nonpositive curvature. Since the algorithm proposed in \cite{bonnabel2011stochastic} relies on a stochastic approximation framework, it necessitates a stepsize that decreases to $0$ over time, that hinders convergence. Furthermore \cite{bonnabel2011stochastic} does not address convergence speed from a theoretical perspective. Consensus on manifolds is also the subject of~\cite{sarlette2009consensus} where the authors embed the manifold in a Euclidean space of larger dimension and turn the consensus problem into an optimization problem in Euclidean space from which they derive a consensus algorithm based on gradient descent. This approach however, is dependent on the embedding of the manifold on which additional conditions are imposed. One can find in \cite{sarlette2009consensus} two examples of manifolds for which such an embedding exists (the rotations group, and Grassmannians) but the specific kind of embedding their result requires might preclude other manifold (positive symmetric matrices for example). The paper~\cite{tron2013riemannian} approaches consensus on manifolds using gradient descent and no embedding of the manifold is needed. The restrictions are placed, instead, on the curvature of the manifold (the sectional curvature is required to be bounded). This covers a broad range of applications. In \cite{tron2013riemannian} a distributed Riemannian gradient descent is used to achieve consensus. The setting of~\cite{tron2013riemannian} is a synchronous one, where each agents update at the same time, as opposed, for example, to random pairwise gossip, where two random agents communicate at each round, as we study in this paper. The main result of~\cite{tron2013riemannian} is that provided a small enough but constant stepsize, the algorithm converges towards a consensus state for manifolds of nonpositive curvature, and converges locally (if the initial set of data is located inside a compact of diameter $<\frac{\pi}{2\sqrt{\kappa}}$) in the case of nonnegative curvature. Interestingly enough, we are going to prove similar results, yet for a distinct setting (pairwise asynchronous -- hence taking randomness into account) and with a distinct approach that does not use gradients, nor differential calculus.

Indeed, in the classical random pairwise gossip case, the computations consists of computing arithmetic means. Generalizing gossip to a broader family of data spaces, naturally leads to consider general metric spaces; and replace arithmetic means by midpoints. However, we shall argue that general metric spaces are too wild to reliably consider midpoints; there could exists many midpoints, or none. Even if midpoints exists and are unique, they could still behave irregularly. Metric spaces with multiple midpoints are numerous; consider for example a circle: opposite points have two midpoints. To construct a metric space lacking midpoints it suffices to delete arbitrary points: consider the previous circle and delete a couple of opposite points. To understand why midpoints could be ill-behaved, consider again a circle parametric by angles; and consider points corresponding to angles $\varepsilon$, $\pi-\varepsilon$, $-\varepsilon$, $\pi + \varepsilon$: $\varepsilon$ and $-\varepsilon$ are close, so are $\pi-\varepsilon$ and $\pi+\varepsilon$, yet the midpoints of ($-\epsilon$, $\pi+\epsilon$) and ($\epsilon$, $\pi-\epsilon$) are far away.

To tame the strange behaviors coming from general metrics, it is natural to study gossip when restricted to Riemannian manifolds, as it has indeed been done (\cite{bonnabel2011stochastic,tron2013riemannian}). However, even if the Riemannian case allows to consider gradients and other differential calculus tools, it hides the simple geometric picture making pairwise gossip work in this setting: namely, comparison theorems. We show that there is a simple tool explaining well the good behavior of pairwise gossip in nonpositive curvature: $CAT(0)$ inequality and that it can even shed some light on gossip in positively curved space (provided we consider instead $CAT(\kappa)$ inequality). This tool is purely metric: no differentials are involved. The benefits of this approach are twofold. Firstly, more general spaces can be given the same analysis as $CAT(\kappa)$ spaces instead of Riemannian manifolds. Secondly, the proofs are simpler because they are purely metric and involve no differential objects from Riemannian Geometry (curvature tensor, Jacobi fields, etc.).

The rest of the paper is organized as follows. Section~\ref{sec:framework} describes the assumptions made on the network and the data. Section~\ref{sec:algo} details the proposed algorithm and formal convergence results are provided in Section~\ref{sec:conv}. Numerical experiments are provided in Section~\ref{sec:simus}.And section~\ref{sec:conclu} concludes the paper.

\section{Framework}\label{sec:framework}

\subsection{Notations}

Assume $V$ is some finite set. We denote by $\mathcal{P}_2(V)$ the set of \emph{pairs} of elements in $V$: $\mathcal{P}_2(V)=\{\{v,w\}: v\not=w\}$. Notice that, by definition, for $v\not =w$, $\{v,w\} =\{w,v\}$ whereas $(v,w)\not=(w,v)$. Throughout the paper, $\mathcal{M}$ will denote a metric space, equipped with metric $d$. Associated with any subset $S\subset\mathcal{M}$, we define its \emph{diameter} $\diam(S) = \sup\{d(s,s'):s,s'\in S\}$. We also define (closed) balls $B(x,r) = \{y\in\mathcal{M}:d(x,y)\leq r\}$. Random variables are denoted by upper-case letters (\emph{e.g.}, $X$, \dots) while their realizations are denoted by lower-case letters (\emph{e.g.} $x$, \dots) Without any further notice, random variables are assumed to be functions from a probability space $\Omega$ equipped with its $\sigma$-field $\mathcal{F}$ and probability measure $\mathbb{P}$; $x = X(\omega)$ denotes the realization associated to $\omega\in\Omega$. For any set $S$ and any subset $A$, $\delta\{A\}$ denotes the indicator function that takes value $1$ on $A$ and $0$ otherwise.

\subsection{Network}

We consider a network of $N$ agents represented by a graph $G = (V,E)$, where $V=\{1,\dots,N\}$ stands for the set of agents and $E$ denotes the set of available communication links between agents. A link $e \in E$ is given by a pair $\{v,w\} \in \mathcal{P}_2(V)$ where $v$ and $w$ are two distinct agents in the network that are able to communicate directly. Note that the graph is assumed undirected, meaning that whenever agent $v$ is able to communicate with agent $w$, the reciprocal communication is also assumed feasible. This assumption makes sense when communication speed is fast compared to agents movements speed. When a communication link $e=\{v,w\}$ exists between two agents, both agents are said to be neighbors and the link is denoted $v\sim w$. We denote by $\mathcal{N}(v)$ the set of all neighbors of the agent $v \in V$. The number of elements in $\mathcal{N}(v)$ is referred to as the \emph{degree} of $v$ and denoted $\deg(v)$.
The graph is assumed to be connected, which means that for every two agents $u,v$ there exists a finite sequence of agents $w_0=u,\dots,w_d=v$ such that:
\[
\forall 0 \leq i \leq d-1: \{w_i,w_{i+1}\} \in E
\]
This means that each two agents are at least indirectly related.

\subsection{Time}
As in \cite{boyd:gossip}, we assume that the time model is asynchronous, \emph{i.e.} that each agent has its own Poisson clock that ticks with a common intensity $\lambda$ (the clocks are identically made), and moreover, each clock is independent from the other clocks. When an agent clock ticks, the agent is able to perform some computations and wake up some neighboring agents. This time model has the same probability distribution than a global single clock ticking with intensity $N\lambda$ and selecting uniformly randomly a single agent at each tick. This equivalence is described, \emph{e.g.} in \cite{boyd:gossip}. From now on, we represent time by the set of integers: for such an integer $k$, time $k$ stands for the time at which the $k^\textrm{th}$ event occurred.

\subsection{Communication}
At a given time $k$, we denote by $V_k$ the agent whose clock ticked and by $W_k$ the neighbor that was in turn awaken. Therefore, at time $k$, the only communicating agents in the whole network are $V_k$ and $W_k$. A single link is then active at each time, hence, at a given time, most links are not used. We assume that $(V_k,W_k)$ are independent and identically distributed and that the distribution of $V_k$ is uniform over the network while the distribution of $W_k$ is uniform in the neighborhood of $V_k$. More precisely, the probability distribution of $(V_k,W_k)$ is given by:
\[
\mathbb{P}[V_k=v,W_k=w] = \begin{cases}\frac 1N\frac1{\deg(v)}&\text{ if }v\sim w\\
0&\text{ otherwise }\end{cases}
\]
Notice that this probability is not symmetric in $(v,w)$. It is going to turn out convenient to also consider directly the link $\{V_k,W_k\}$, forgetting which node was the first to wake up and which node was second. In this case $\mathbb{P}[\{V_k,W_k\}=\{v,w\}]$ is of course symmetric in $(v,w)$. One has:
\[
\mathbb{P}[\{V_k,W_k\}=\{v,w\}] = \begin{cases}\frac 1N(\frac1{\deg(v)}+\frac1{\deg(w)})&\text{ if }v\sim w\\
0&\text{ otherwise }\end{cases}
\]
The communication framework considered here is standard \cite{boyd:gossip}.

\subsection{Data}
Each node $v \in V$ stores data represented as an element $x_v$ belonging to some space $\mathcal{M}$. More restrictive assumptions on $\mathcal{M}$ will follow (see section~\ref{sec:catk}). Initially each node $v$ has a value $x_v(0)$ and $X_0=(x_1(0),\dots,x_N(0))$ is the tuple of initial values of the network. We focus on iterative algorithms that tend to drive the network to a \emph{consensus state}; meaning a state of the form $X_{\infty}=(x_{\infty},\dots,x_{\infty})$ with: $x_{\infty} \in \mathcal{M}$. We denote by $x_v(k)$ the value stored by the agent $v \in V$ at the $k$-th iteration of the algorithm, and $X_k=(x_1(k),\dots,x_N(k))$ the global state of the network at instant $k$. The general scheme is as follows: network is in some state $X_{k-1}$; agents $V_k$ and $W_k$ wake up, communicate and perform some computation to lead the network to state $X_k$.

\section{Algorithm}\label{sec:algo}

At each count of the virtual global clock one node $v$ is selected uniformly randomly from the set of agents $V$. The node $v$ then randomly selects a node $w$ from $\mathcal{N}(v)$. Both node $v$ and $w$ then compute and update their value to $\langle\frac{x_v+x_w}{2}\rangle$.
\begin{varalgorithm}{Random Pairwise Midpoint}
\begin{algorithmic}
\STATE \textbf{Input}: a graph $G=(V,E)$ and the initial nodes configuration $X_{v}(0), v \in V$

\FORALL {$k>0$}
		\STATE At instant $k$, uniformly randomly choose a node $V_k$ from $V$ and a node $W_k$ uniformly randomly from $\mathcal{N}(V_k)$.
            \STATE Update:
		\STATE       $X_{V_k}(k)=\left\langle\frac{X_{V_k}(k-1)+X_{W_k}(k-1)}{2}\right\rangle$
            \STATE       $X_{W_k}(k)=\left\langle\frac{X_{V_k}(k-1)+X_{W_k}(k-1)}{2}\right\rangle$
            \STATE
$X_v(k) = X_v(k-1)$ for $v\not\in\{V_k,W_k\}$
\ENDFOR
\end{algorithmic}
\caption{}\label{algo:gpapa}
\end{varalgorithm}

\begin{rem}
Please note that the previous algorithm is well defined in the case where data belongs to some $CAT(0)$ space thanks to proposition~\ref{prop:cat0midpoints}. Otherwise, midpoints are not necessarily well-defined; and the algorithm should read compute any midpoint between $X_{V_k}(k-1)$ and $X_{W_k}(k-1)$, if there exists some. However, we are going to see in the next sections that, in this case, the algorithm might fail to converge to a consensus.
\end{rem}

\section{Convergence results}
\label{sec:conv}

In order to study convergence we recall the following assumptions, already explained in section~\ref{sec:framework}.
\begin{assum}
\label{assum:common}
\mbox{}
\begin{enumerate}
\item $G=(V,E)$ is connected
\item $(V_k,W_k)_{k\geq 0}$ are i.i.d random variables, such that:
  \begin{enumerate}
    \item $(V_k,W_k)$ is independent from $X_0,\dots, X_{k-1},(V_0,W_0),\dots, (V_{k-1},W_{k-1})$,
    \item $\mathbb{P}[\{V_0,W_0\}=\{v,w\}] = \frac1N(\deg^{-1}(v) + \deg^{-1}(w))\delta\{v\sim w\}$
  \end{enumerate}
\end{enumerate}
\end{assum}

\subsection{$CAT(0)$ spaces}

In this subsection we make the following assumption.
\begin{assum}
\label{assum:cat0}
$(\mathcal{M},d)$ is a \emph{complete} $CAT(0)$ metric space.
\end{assum}


We now define the \emph{disagreement function}.
\begin{defn}
	\label{def:disag}
	Given a configuration $x = (x_1,\dots,x_N)\in\mathcal{M}^N$ the disagreement function
    \[
    \Delta(x) = \sum_{v\sim w\atop\{v,w\}\in \mathcal{P}_2(V)}(\deg(v)^{-1}+\deg(w)^{-1})d^2(x_v,x_w)
    \]
\end{defn}
Function $\Delta$ measures how much disagreement is left in the network. Indeed, since the network is connected, $\Delta$ is $0$ if and only if the network is at consensus. It would be a graph Laplacian in the Euclidean setting. The normalizing term involving degrees gives less weight to more connected vertices, since they are more likely to be solicited by neighbors; in order to give equal weight to each edge in the graph. This normalization will turn out to be convenient in the analysis. Another important function is the \emph{variance} function.
\begin{defn}
\label{def:sigmasq}
Given a configuration $x = (x_1,\dots,x_N)\in\mathcal{M}^N$, the \emph{variance} function is defined as:
\[
\sigma^2(x) = \frac1{N}\sum_{\{v,w\}\in\mathcal{P}_2(V)}d^2(x_v,x_w)
\]
\end{defn}
\begin{rem}
The $\frac1{N}$ normalizing constant accounts for the fact that when $d$ is the Euclidean distance then $\sigma^2(x)$ equals $\sum_{v\in V}\|x_v-\bar x\|^2$, with $\bar x = \frac1N\sum_{v\in V} x_v$.
\end{rem}
The next proposition measures the average decrease of variance at each iteration.
\begin{prop}
\label{prop:L}
Under Assumptions~\ref{assum:common} and \ref{assum:cat0}, for $X_k$ given by Algorithm~\ref{algo:gpapa}, the following inequality holds, for every $k\geq1$.
\[ \mathbb{E}[\sigma^2(X_{k})-\sigma^2(X_{k-1})] \leq -\frac 12\mathbb{E}[\Delta(X_{k-1})] \]
\end{prop}
\begin{proof}
Taking into account that at round $k$, two nodes woke up with indices $V_{k}$ and $W_{k}$, it follows that:
\[
N(\sigma^2(X_{k}) - \sigma^2(X_{k-1})) = -d^2(X_{V_k}(k-1),X_{W_k}(k-1)) + \sum_{u\in V\atop u\not= V_k,u\not=W_k} T(V_k,W_k,u)
\]
where $T(V_k,W_k,u) = 2d^2(X_{u}(k),M_{k}) - d^2(X_{u}(k),X_{V_k}(k-1))-d^2(X_{u}(k),X_{W_k}(k-1))$ and $M_k$ denotes the midpoint $\langle\frac{X_{V_k}(k-1)+X_{W_k}(k-1)}2\rangle$. Notice that $X_{U_{k}}(k) = X_{V_{k}}(k) = M_k$.
Now, using the CAT(0) inequality, one has:
\[
N(\sigma^2(X_{k}) - \sigma^2(X_{k-1})) \leq \frac{N}2d^2(X_{V_k}(k-1),X_{W_k}(k-1))\;.
\]
Taking expectations on both sides and dividing by $N$ gives:
\[
\mathbb{E}[\sigma^2(X_{k}) - \sigma^2(X_{k-1})] \leq -\frac{1}{2}\mathbb{E}[d^2(X_{V_k}(k-1),X_{W_k}(k-1))]
\]
Recalling that $\mathbb{P}[\{V_k,W_k\}=\{u,v\}]=\frac{1}{\deg u}+\frac{1}{\deg v}$ when $u\sim v$ and $0$ otherwise, and that $(V_k,W_k)$ are independent from $X_{k-1}$, one can deduce:
\[
\mathbb{E}[d^2(X_{V_k}(k-1),X_{W_k}(k-1))] = \mathbb{E}[\Delta(X_{k-1})]
\]

\end{proof}

\begin{prop}
\label{prop:diamlapl}
Assume $G=(V,E)$ is an undirected connected graph, there exists a constant $C_G$ depending on the graph only such that:
\[
\forall x\in\mathcal{M}^N, \mskip 20mu \frac{1}{2N}\Delta(x)\leq \sigma^2(x) \leq C_G\Delta(x)
\]
\end{prop}
\begin{proof}
First:
\begin{eqnarray*}
\Delta(x) & = & \sum_{v\sim w} (\deg(v)^{-1} + \deg(w)^{-1})d^2(x_v,x_w) \\
& \leq & 2\sum_{v\sim w}d^2(x_v,x_w) \\
& \leq & 2\sum_{\{v,w\}\in\mathcal{P}_2(V)} d^2(x_v,x_w) = 2N\sigma^2(x)
\end{eqnarray*}

For the second inequality, consider $v\not=w$ two vertices in $V$, not necessarily adjacent. Since $G$ is connected, there exists a path $u_0 = v$, \dots, $u_l=w$ such that $u_i\sim u_{i+1}$. Then, using Cauchy-Schwartz inequality:
\[
d(v,w)^2\leq l\sum_{i=0}^{l-1} d^2(u_i,u_{i+1}) \leq 2\deg(G)\diam(G)\sum_{i=0}^{l-1}(\deg(u_i)^{-1} + \deg(u_{i+1})^{-1})d^2(u_i,u_{i+1})
\]
where $\deg(G)$ denotes the \emph{maximum degree} $\max\{\deg(v):v\in V\}$ and $\diam(G)$ the \emph{diameter} of $G$. Hence taking $C_G = (N-1)\deg(G)\diam(G)$, one recover the sought inequality.
\end{proof}
\begin{rem}
Both functions $\Delta$ and $\sigma^2$ measure disagreement in the network, $\Delta$ takes into account the graph connectivity while $\sigma^2$ does not. The previous result shows that $\Delta$ and $\sigma^2$ are nonetheless equivalent up to multiplicative constants.
\end{rem}

We now state a first convergence result.
\begin{thm}[Almost-sure convergence to consensus]
  \label{the:cons}
Under Assumptions~\ref{assum:common} and \ref{assum:cat0}, there exists a random variable $X_\infty=(X_{\infty,v})_{v\in V}$, such that: (i) almost surely, $\forall (v,w)\in V^2$, $X_{\infty,v} = X_{\infty,w}$, \emph{i.e.} $X_\infty$ takes consensus values, and (ii) $X_k$ converges almost surely to $X_\infty$.
\end{thm}
\begin{proof}
Let us first show that $\Delta(X_k)$ converges almost surely to $0$.
From proposition~\ref{prop:L}, $\mathbb{E}[\sigma^2(X_k)]$ is nonincreasing; which implies again from proposition~\ref{prop:L}:
\[
\sum_k\mathbb{E}[\Delta(X_k)] < 2\sigma^2(X_0) < \infty
\]
Hence, $\sum_k \Delta(X_k)$ has a finite expectation and $\Delta(X_k)$ converges almost surely to $0$. Therefore, using the first inequality in proposition~\ref{prop:diamlapl}, $\sigma^2(X_k)$ converges to $0$. As a direct consequence, the diameter $\max\{d(X_v(k),X_w(k)):(v,w)\in V^2\}$ also tends to $0$ when $k$ goes to $\infty$.
Now denote by $S_k$ the set $\{X_v(k):v\in V\}$ and by $\conv(S_k)$ its \emph{convex hull}. One has $\diam(\conv(S_k))\leq 2\diam(S_k)$: every ball centered at $X_v(k)$ with radius $\diam(S_k)$ is a convex set containing $S_k$ and hence $\conv(S_k)$. Moreover, using the definition of convexity, one has $S_{k+1}\subset\conv(S_k)$. Therefore $\conv(S_{k})$ form a family of nested closed sets with diameter converging to $0$. It is an easy result that in a complete metric space, the intersection of a family of nested closed subsets with diameter converging to $0$ is reduced to a singleton.
\end{proof}
Actually the previous proof can be adapted to give information on the convergence speed of the algorithm. Let us first prove an elementary lemma.
\begin{lem}
\label{lem:gronwal}
Assume $a_n$ is a sequence of nonnegative numbers such that $a_{n+1} - a_n\leq -\beta a_n$ with $\beta>0$. Then,
\[
\forall n\geq 0, \mskip 20mu a_n \leq a_0\exp(-\beta n)
\]
\end{lem}
\begin{proof}
Indeed if $l_n = \log{a_n}$, then $l_{n+1} - l_n \leq \log(1-\beta)\leq -\beta$. Hence $l_n\leq l_0-\beta n$. Taking exponential on both side gives the expected result.
\end{proof}
We are now in a position to prove the following result:
\begin{thm}[Convergence speed]
	\label{the:speed}
Let $X_k=(x_{1}(k),...,x_{N}(k))$ denote the sequence of random variables generated by Algorithm~\ref{algo:gpapa}, under Assumptions~\ref{assum:common} and \ref{assum:cat0}, there exists $L<0$ such that,
\[
		\limsup_{k\to\infty}\frac{\log\mathbb{E}\sigma^2(X_k)}k \leq L
\]
\end{thm}
\begin{proof}
Denote by $a_n = \mathbb{E}\sigma^2(X_k)$. Using the same argument as in the proof of theorem~\ref{the:cons} and proposition~\ref{prop:diamlapl}, we know that there exists a constant $L>0$ such that $a_{n+1} - a_n\leq La_n$. We conclude using lemma~\ref{lem:gronwal}.
\end{proof}
\begin{rem}
  Using Proposition~\ref{prop:diamlapl} it is straightforward to see that an analogous inequality holds for $\limsup_{k\to\infty}\log\mathbb{E}\Delta(X_k)/k$.
\end{rem}
\begin{rem}
What we have shown so far, is that for $CAT(0)$ spaces both convergence and convergence speed are similar to the Euclidean case; yet the proof techniques only rely on metric comparisons, whereas spectral techniques are mainly used in the Euclidean case (\emph{e.g.} \cite{boyd:gossip}).
\end{rem}

We now turn to the case of positively curved spaces.
\subsection{$CAT(\kappa)$ spaces with $\kappa > 0$}\label{subsec:positive}

In this section, we replace Assumption~\ref{assum:cat0} by the following:
\begin{assum}
\label{assum:catk}
\mbox{}
\begin{enumerate}
\item $\kappa>0$
\item
$(\mathcal{M},d)$ is a \emph{complete} $CAT(\kappa)$ metric space.
\item\label{assum:catk3}
$\diam(\{X_v(0):v\in V\}) < r_\kappa$
\end{enumerate}
\end{assum}
By proposition~\ref{prop:catk} we are ensured that Algorithm~\ref{algo:gpapa} is well-defined. Indeed, by convexity of balls with radius smaller than $r_\kappa$ points will remain within distance less than $r_\kappa$ of each other. Moreover midpoints are well-defined and unique since $r_\kappa<D_\kappa$.

The trick used to study $CAT(\kappa)$ configurations is to replace distance $d(x,y)$ by: $\chi_\kappa(d(x,y))$ with $\chi_\kappa = 1 - C_\kappa$ being pointwise nonnegative. We adapt the definitions used in the $CAT(0)$ setting as follows:
\begin{defn}
for $x\in\mathcal{M}^n$ define:
\[\Delta_\kappa(x)=\frac{1}2\sum_{v\sim w\atop\{v,w\}\in E}(\deg(v)^{-1}+\deg(w)^{-1})\chi_{\kappa}(d(x_v,x_w)) \]
\[\sigma_\kappa^2(x)=\frac2{N}\sum_{\{v,w\}\in\mathcal{P}_2(V)}\chi_{\kappa}(d(x_v,x_w))\]
\end{defn}
One can remark that for all $k \geq 0$, $(v,w)\in V^2$: $\sigma_\kappa^2(X_k) \geq 0$ and $\Delta_\kappa(X_k) \geq 0$. Notice that $\sigma_\kappa^2(x)=0$ implies that for all $\{v,w\}\in \mathcal{P}_2(V)$: $\chi_\kappa(d(v,w))=0$; and, since $0\leq d(v,w) \leq \frac{\pi}{2\sqrt{\kappa}}$, it implies that $d(v,w)=0$, hence the system is in a consensus state. Moreover, when $\kappa \to 0$, $\Delta_\kappa\to\Delta$ and $\sigma^2_\kappa\to\sigma^2$.

The following proposition is a direct consequence of lemma~\ref{lem:trigo}.
\begin{prop}
\label{prop:ineqk}
Under Assumption~\ref{assum:catk}, for any triangle $\Delta(pqr)$ in $\mathcal{C}$ where $m$ is the midpoint of $[p,q]$ we have:
\[\chi_\kappa(d(m,r)) \leq \frac{\chi_\kappa(d(p,r))+\chi_\kappa(d(q,r))}{2}\]
\end{prop}
\begin{proof}
From lemma~\ref{lem:trigo} we get:
\[
C_\kappa(d(p,r))+C_\kappa(d(q,r))-2C_\kappa(d(m,r)) \leq 2C_\kappa(d(m,r))C_\kappa(d(p,q))-2C_\kappa(d(m,r))
\]
Since $\max\{d(m,r),d(p,q)\}<\frac{\pi}{2\sqrt{\kappa}}$ we have:
$0 \leq C_\kappa(d(p,q)) \leq 1$ and $0 \leq C_\kappa(d(m,r)) \leq 1$
Which means that: $2C_\kappa(d(m,r))C_\kappa(d(p,q))-2C_\kappa(d(m,r))\leq 0$.
And thus:
\[
2\chi_\kappa(d(m,r)) \leq \chi_\kappa(d(p,r))+\chi_\kappa(d(q,r))
\]
\end{proof}

With this result it is now possible to prove using the same reasoning as in proposition~\ref{prop:L}. The techniques are the same but the details differ slightly. For the sake of completeness, we give the details below.
\begin{prop}
\label{prop:L2}
\[ \mathbb{E}[\sigma_\kappa^2(X_{k+1})-\sigma_\kappa^2(X_k)] \leq -\frac{1}{N}\mathbb{E}\Delta_\kappa(X_k) \]
\end{prop}
\begin{proof}
At round $k$, two nodes woke up with indices $V_{k}$ and $W_{k}$, it follows that:
\[
N(\sigma_\kappa^2(X_{k}) - \sigma_\kappa^2(X_{k-1})) = -\chi_{\kappa}(d(X_{V_k}(k-1),X_{W_k}(k-1))) + \sum_{u\in V\atop u\not= V_k,u\not=W_k} T_\kappa(V_k,W_k,u)
\]
Where $T_\kappa$ is defined as:
\[
T_\kappa(V_k,W_k,u) = 2\chi_\kappa(d(X_{u}(k),M_{k})) - \chi_\kappa(d(X_{u}(k),X_{V_k}(k-1)))-\chi_\kappa(d(X_{u}(k),X_{W_k}(k-1)))\;.
\]
Now, using the inequality of proposition~\ref{prop:ineqk}, one gets that $T_\kappa(V_k,W_k,u) \leq 0$ and:
\[
N(\sigma_\kappa^2(X_{k}) - \sigma_\kappa^2(X_{k-1})) \leq \chi_\kappa(d(X_{V_k}(k-1),X_{W_k}(k-1)))\;.
\]
Taking expectations on both sides and dividing by $N$ gives:
\[
\mathbb{E}[\sigma_\kappa^2(X_{k}) - \sigma_\kappa^2(X_{k-1})] \leq -\frac1N\mathbb{E}[\chi_\kappa(d(X_{V_k}(k-1),X_{W_k}(k-1)))]
\]
Using similar reasoning as in the proof of proposition~\ref{prop:L} we have:
\[
\mathbb{E}[\chi_\kappa(d(X_{V_k}(k-1),X_{W_k}(k-1)))] = \mathbb{E}[\Delta_\kappa(X_{k-1})]
\]
Which yields:
\[ \mathbb{E}[\sigma_\kappa^2(X_{k+1})-\sigma_\kappa^2(X_k)] \leq -\frac1N\mathbb{E}\Delta_\kappa(X_k)
\]
\end{proof}
\begin{rem}
  Notice the constant $1/N$ is in the right hand which differs from the case of nonpositive curvature (compare with Proposition~\ref{prop:L}).
\end{rem}
In order to derive a convergence result we need an analogous result to Proposition~\ref{prop:diamlapl} for $CAT(\kappa)$ spaces.
\begin{prop}
\label{prop:diamlaplk}
Assume $G=(V,E)$ is an undirected connected graph, there exists a constant $C_\kappa$ depending on the graph only such that:
\[
\forall x\in\mathcal{M}^N, \mskip 20mu \frac{\kappa}{N\pi^2}\Delta_\kappa(x)\leq \sigma_\kappa^2(x) \leq C_\kappa\Delta_\kappa(x)
\]
\end{prop}
\begin{proof}
One has: $\frac{2\kappa}{\pi^2}x^2 \leq \chi_\kappa(x) \leq \frac{\kappa}{2}x^2$ when $0\leq x<\frac{\pi}{2\sqrt{\kappa}}$. Hence, under Assumption~\ref{assum:catk}, $\chi_\kappa$ and $d$ are equivalent. The result then follows from Proposition~\ref{prop:diamlaplk}.
\end{proof}
All the tools to show almost-sure convergence and speed are in place. The proofs of the following two results are exactly the same than in the $CAT(0)$ case, provided $\Delta$ and $\sigma$ are replaced by $\Delta_\kappa$ and $\sigma_\kappa$.
\begin{thm}
\label{the:consk}
Let $X_k=(X_{1}(k),...,X_{N}(k))$ denote the sequence generated by Algorithm~\ref{algo:gpapa}, then under Assumptions~\ref{assum:common} and \ref{assum:catk}, there exists a random variable $X_\infty$ taking values in the consensus subspace, such that $X_k$ tends to $X_\infty$ almost surely.
\end{thm}
\begin{thm}
	\label{the:speedk}
Let $X_k=(x_{1}(k),...,x_{N}(k))$ denote the sequence of random variables generated by Algorithm~\ref{algo:gpapa}; under Assumptions~\ref{assum:common} and \ref{assum:catk}, there exists $L<0$ such that,
\[
\limsup_{k\to\infty}\frac{\log\mathbb{E}\Delta_\kappa(X_k)}k \leq L
\]
\end{thm}

These results show that -- \emph{provided all the initial points are close enough from each other}, this is detailed by Assumption~\ref{assum:catk}.\ref{assum:catk3} -- the situation is the same as in nonpositive curvature, namely, almost sure convergence taking place at least exponentially fast. Notice that, by contrast, there are no constraints on the initialization, for the result to hold true in $CAT(0)$. Notice also that the radius involved in Assumption~\ref{assum:catk}.\ref{assum:catk3} depends on the curvature upper bound $\kappa$ and ensures convexity of corresponding balls. It gives a hint that convexity plays an important role in the behavior of the algorithm, which is not surprising, since the algorithm basically amounts to take random midpoints.

\section{Numerical Simulations}
\label{sec:simus}

In this section we simulate Algorithm~\ref{algo:gpapa} through four examples. The first example is the space of covariance matrices; it is a Hadamard manifold (\emph{i.e.}, a complete, simply connected manifold with nonpositive sectional curvature, see, \emph{e.g.} \cite[Chap XI.3]{lang}). The second is the metric graph, (a complex of $(0,1]$ segments), which is a $CAT(0)$ metric space with no differential structure. The other two examples are of $CAT(\kappa)$ spaces with $\kappa>0$. They are the three dimensional unit sphere $S^2$ and three dimensional rotation matrices $SO(3)$.

When one of the above mentioned spaces happens to be stable by addition and multiplication by a scalar (it is the case for positive definite matrices), we compare the performance of \midgossip with that of the \arithgossip. In order to clarify between the two algorithms when they can both be used; we use the term \midgossip for Algorithm~\ref{algo:gpapa} and the term \arithgossip for the classical random pairwise algorithm $X_{n+1,v} = X_{n+1,w} = \frac12(X_{n,v}+X_{n,w})$ \cite{boyd:gossip} which is equivalent to \midgossip when the distance is the Euclidean one.

The results of these comparisons, as we shall see, might depend on the distance function used to define the disagreement function, or equivalently, the variance function.

\subsection{Positive definite matrices}

The scenario in this experiment is the following. Each sensor in a network estimates a covariance matrix for some observed multivariate process. Then the network seeks a consensus on these covariance matrices. We implemented the proposed algorithm using known facts from the geometry of positive definite matrices $\Pos(n)$, \cite[chap. 12]{lang}. $\Pos(n)$ is equipped with distance
\[
d(M,N)^2 = \tr\{\log(N^{-1/2}MN^{-1/2})\log(N^{-1/2}MN^{-1/2})^T\} = \|\log(MN^{-1})\|^2\,,
\]
and
\[
\langle\frac{M+N}2\rangle = M^{1/2}(M^{-1/2}NM^{-1/2})^{1/2}M^{1/2}\;.
\]
Using this distance, which comes from a Riemannian metric, $\Pos(n)$ is a Hadamard manifold~\cite[p.326]{lang}, see also \cite{barbaresco:2013} for an in-depth presentation, and as such, it is a $CAT(0)$ space\cite[prop 3.4, p.311]{lang}. Using the previous relations, it is straightforward to implement the \midgossip algorithm and compute $\log\sigma^2\big(M(n)\big)$ at each iteration $n$; where $M(n)=(M_1(n),\dots, M_N(n))$ denotes the tuple of positive definite matrix held by the agents $1\leq v\leq N$ at time $n$. Regarding the initialization step, we generate $N$ iid matrices $M_v(0)$, following a Wishart distribution on $q\times q$ positive definite matrices with parameters $(q,1)$, \emph{i.e.}, as $\sum_{k=1}^q X_{k,v}X_{k,v}^T$ where $X_{k,v}\sim\mathcal{N}(0,I_q)$ are independent standard multivariate Gaussian vectors of dimension $q$ (in this numerical experiment $q=3$ and $N=30$). Regarding the network, we experiment with both the complete graph $K_N$ and the path graph $P_N$ ($V=\{1\dots N\}$, $\{i,i+1\} \in E$ for $i\in \{1,\dots N-1\}$). The complete graph mixes information fast, while the path graph does not. It is interesting to compare the results obtained in both cases. are displayed in Figure~\ref{fig:covariance} (for complete graph) and~\ref{fig:CovRiemPn} (for path graph). Note that the algorithm is very close to the one proposed in \cite{bonnabel2011stochastic} which consists in the iterations $M^{1/2}(M^{-1/2}NM^{-1/2})^{\gamma_n}M^{1/2}$ where $\gamma_n$ is a sequence of stepsize such that $\sum_n \gamma_n = +\infty$ and $\sum_n\gamma_n^2<\infty$. In particular, stepsize $\gamma_n$ should go to $0$ while in our case it is kept constant at $1/2$. The full and dashed curves in figure~\ref{fig:CompCovGrad} represent the function $\log(\sigma^2_n)$ for respectively the stochastic gradient descent method (implemented with a decreasing step size $\gamma_n=\frac{1}{n}$) and midpoint gossip algorithm; the initialization and graph used for both algorithms being the same (complete graph), the two curves can be compared so as to deduce that while the consensus midpoint algorithm leads to exponential convergence, the $\log(\sigma^2_n)$ curve for the gradient descent method seems to converge slower. Actually the fact that it converges slower is coherent with stochastic approximation with decreasing stepsize. Indeed, it is known that, in the Euclidean setting \cite[chap. 10]{kushner:yin:1997}, for stepsize $\gamma_n$, the speed of convergence is of order ${\gamma_n}^{-1/2}$.


It is also interesting in this case to make a comparison for positive definite matrices between the Midpoint gossip algorithm and the Euclidean arithmetic gossip. In figure~\ref{fig:CompCovRiem} we plot $n\mapsto \sigma^2_n$ where $n$ is the number of iterations and $\sigma^2$ is the sum of "non Euclidean" distances squared. The result suggests that the Arithmetic gossip algorithm has a slight advantage over midpoint gossip in terms of convergence speed. However, if we plot $n\mapsto \sigma^2_{n,E}$ where $\sigma^2_{n,E}= \frac{1}{N}\sum_{\{i,j\}\in \mathcal{P}_2(V)}||x_i(n)-x_j(n)||_F^2$ and $||.||_F$ is the Frobenius Euclidean norm, the opposite seems to be true, as shown in figure~\ref{fig:CompCovEuc} midpoint algorithm performs slightly better. The midpoint gossip algorithm converges faster than arithmetic gossip when the variance is expressed in Euclidean distances.

\begin{figure}[ht!]
  \centering
  \includegraphics[width=0.8\linewidth]{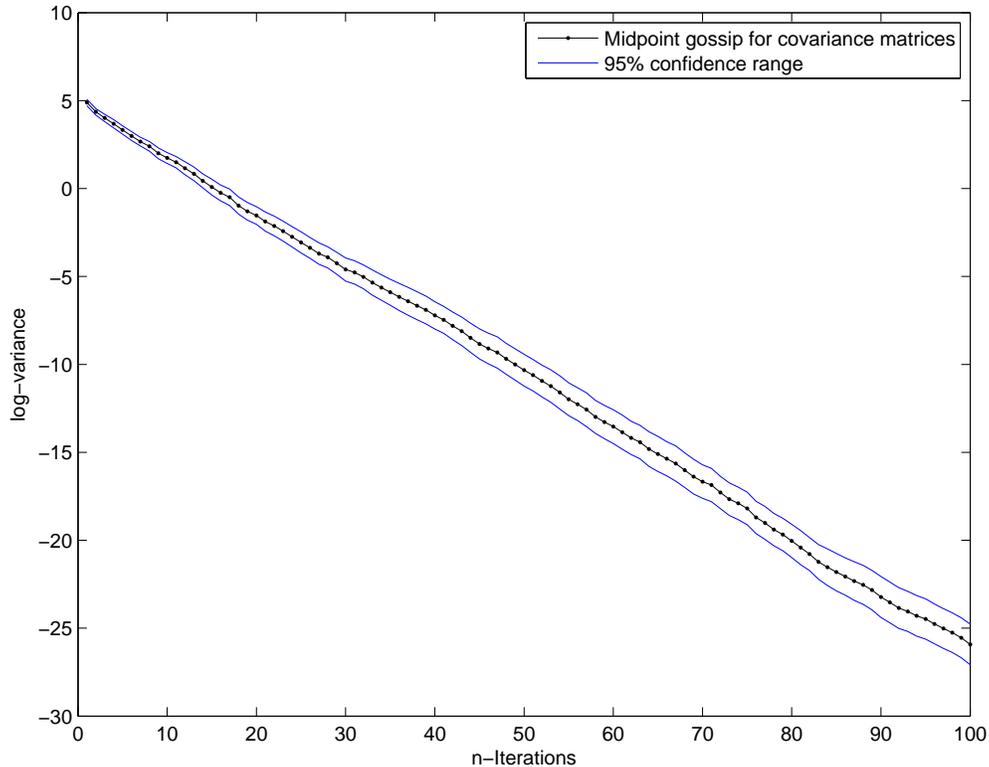}
  \caption{Plot of $n\mapsto \log \sigma^2_n$ for the positive definite matrices space; the underlying network is the complete graph $K_n$. Because of the stochastic nature of the algorithm, $50$ simulations are done and we plot $\log\sigma^2_n$ as well as a confidence domain which contains $95\%$ of the simulated curves. The variance function behaves like an exponential, in accordance with the prediction of theorem~\ref{the:speed}.}
  \label{fig:covariance}
  \end{figure}

\begin{figure}[ht!]
  \centering
  \includegraphics[width=0.8\linewidth]{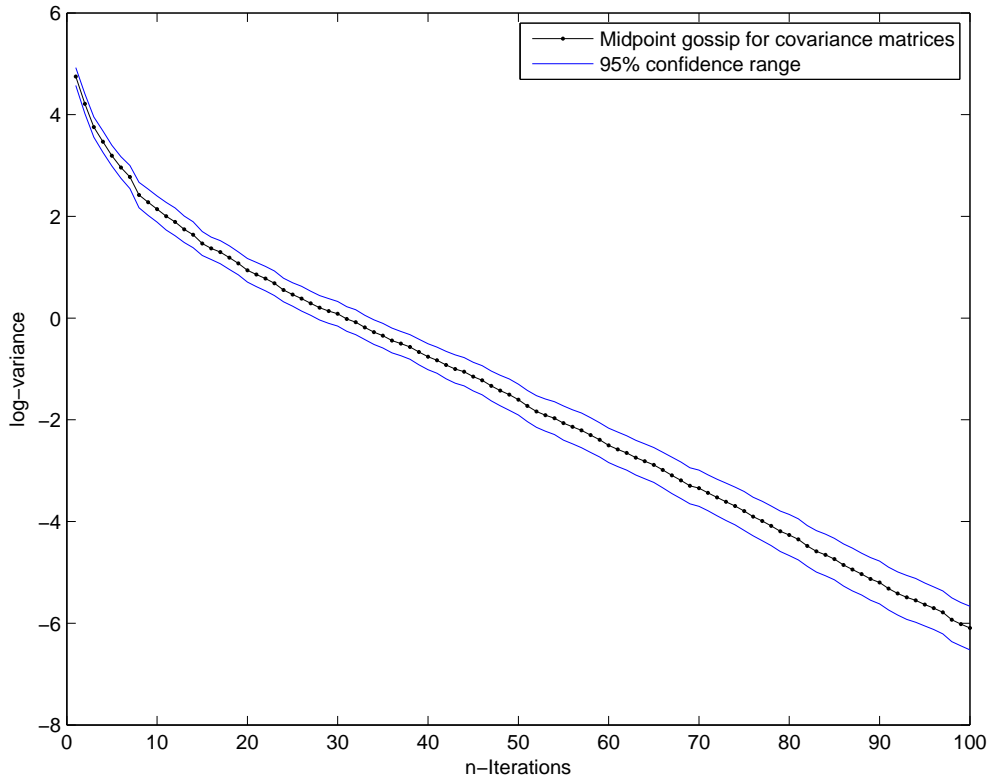}
  \caption{Plot of $n\mapsto \log \sigma^2_n$ for the positive definite matrices manifold using the path graph $P_n$, Because of the stochastic nature of the algorithm, $50$ simulations are done and we plot $\log\sigma^2_n$ as well as a confidence domain which contains $95\%$ of the simulated curves. We see that the variance asymptotically behaves like an exponential, in accordance with the prediction of theorem~\ref{the:speed}. The convergence tough is much slower than that of the complete graph, with a smaller slope. The connectivity of the graph plays an important role in determining the speed of convergence.}
  \label{fig:CovRiemPn}
\end{figure}

\begin{figure}[ht!]
  \centering
  \includegraphics[width=0.8\linewidth]{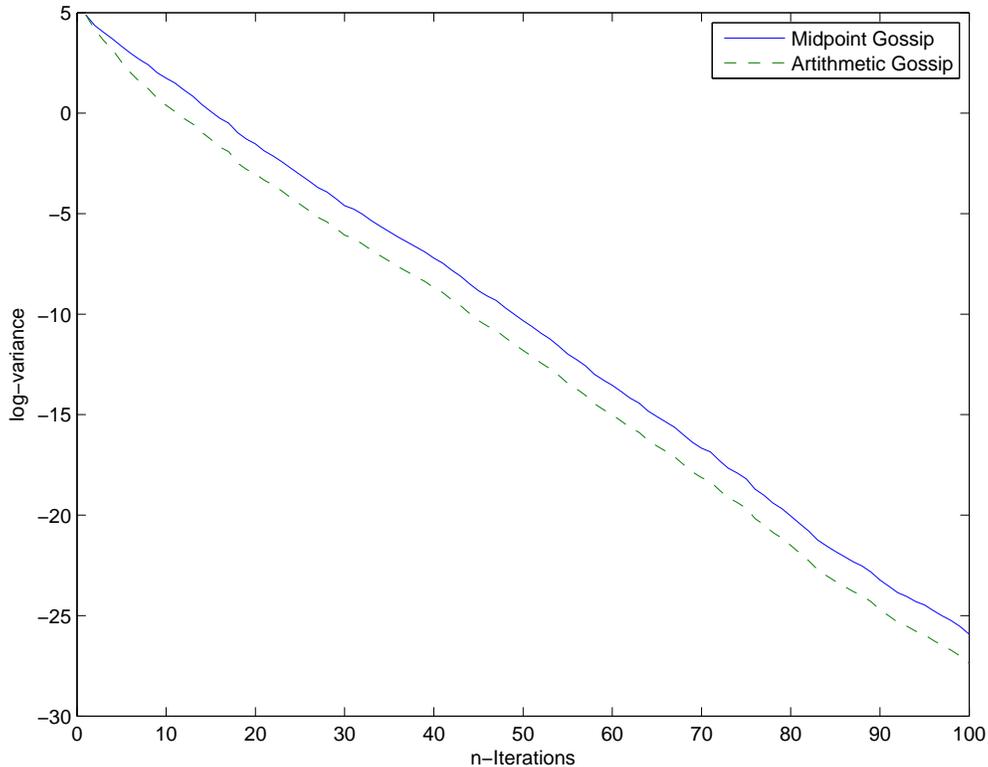}
  \caption{Plot of $n\mapsto \log \sigma^2_n$ (non euclidean distances) for the positive definite matrices manifold, the full curve represents the midpoint gossip algorithm, while the dashed curve represents the classical gossip based on arithmetic averaging. The arithmetic gossip seems to converge faster.}
  \label{fig:CompCovRiem}
\end{figure}

\begin{figure}[ht!]
  \centering
  \includegraphics[width=0.8\linewidth]{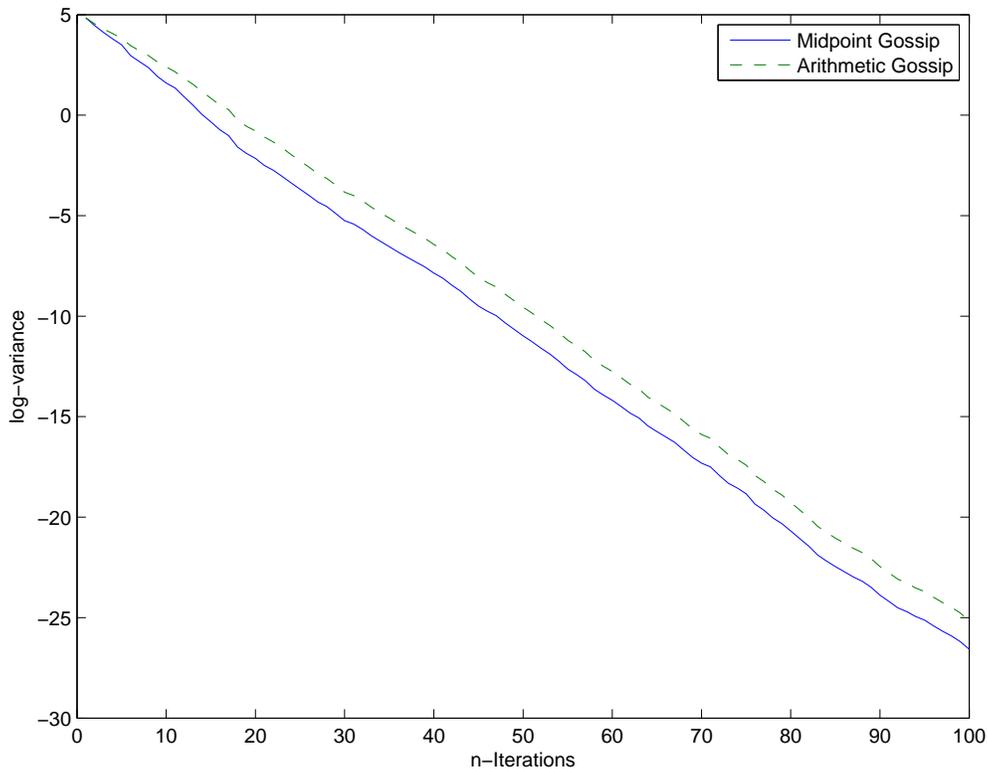}
  \caption{Plot of $n\mapsto \log \sigma^2_{n,E}$ (using euclidean norm) for the positive definite matrices manifold, the full curve represents the Pairwise Midpoint Algorithm, while the dashed curve represents the classical Euclidean gossip. The midpoint gossip seems to have faster convergence.}
  \label{fig:CompCovEuc}
\end{figure}

\begin{figure}[ht!]
  \centering
  \includegraphics[width=0.8\linewidth]{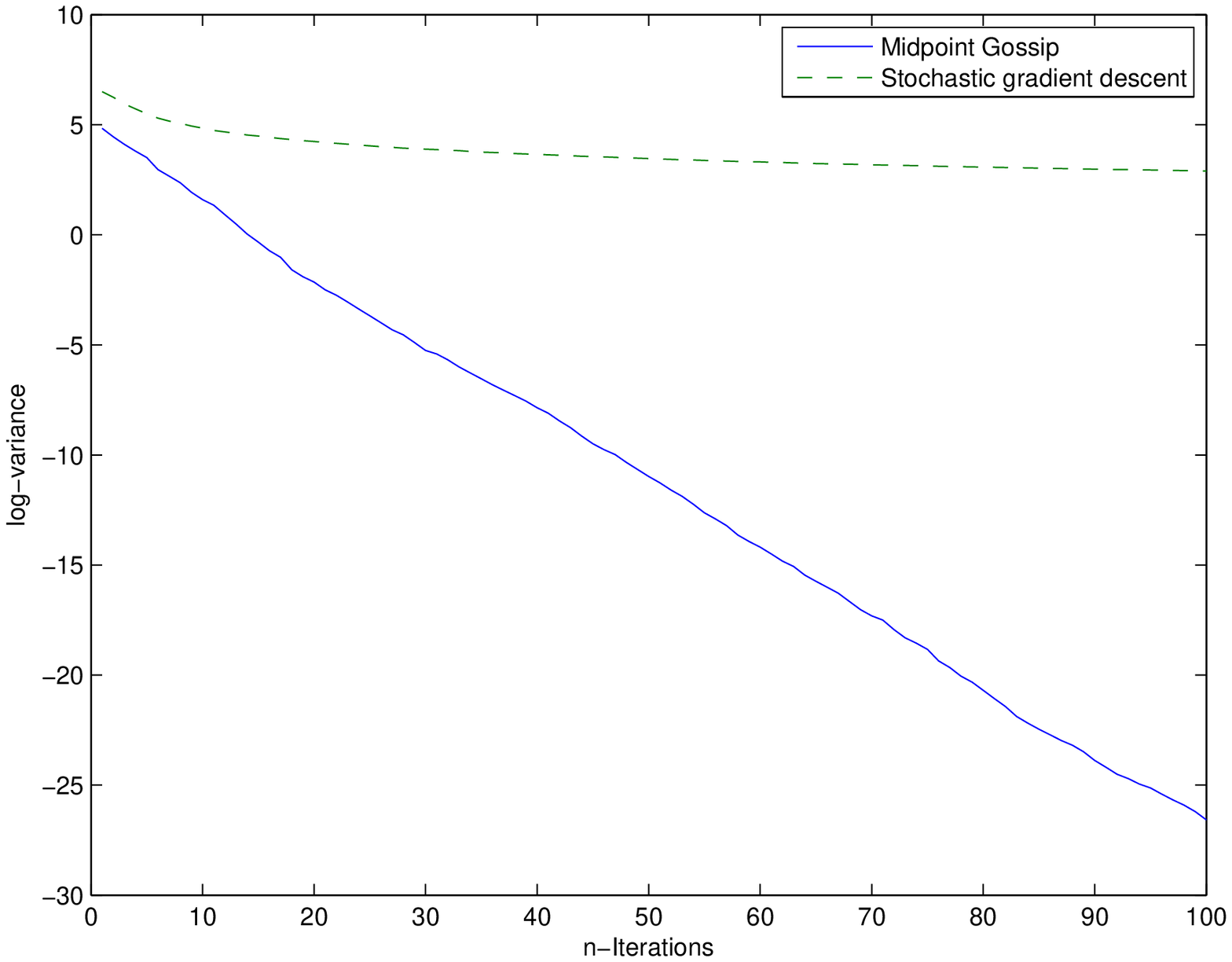}
  \caption{Plot of $n\mapsto \log\sigma^2_n$ for the positive definite matrices manifold. The full curve represents the Riemannian midpoint gossip while the dashed curve represents the stochastic gradient descent method applied to the function $\sigma^2$. Convergence is exponential in the first case while it is not for the second.}
  \label{fig:CompCovGrad}
\end{figure}

\subsection{The metric graph associated with free Group $F_2$}
Consider a network $G=(V,E)$ of robotic arms capable of performing two types of rotations $R_1$ and $R_2$, of distinct rotation axes $\Delta_1$ and $\Delta_2$. After being assigned an axis of rotation, an arm rotates continuously around that axis until it reaches its target rotation angle or gets interrupted. At an initial time, all arms are in an identical position. After that they evolve separately. To recollect them in a common position, they unwind all their movements, provided they kept the whole history. Here, we argue that a less costly procedure could be applied. We apply Algorithm~\ref{algo:gpapa} on a convenient state space that we describe below, to drive the arms near a consensus position, in a completely distributed and autonomous fashion.


We assume furthermore that $R_1$ and $R_2$ are chosen in ``generic position''. By that, we mean that they are realizations of independent and uniform random variables on $SO_3$ (\emph{i.e.} according to the Haar measure on $SO_3\times SO_3$). As such, it is known, \cite{epstein1971almost}, that they are almost surely algebraically independent, \emph{i.e.}, if $\exists n> 0$, $(k_1,\dots ,k_n) \in \{-1,1\}$ and $i_1\dots i_n\in \{1,2\}$ such that: $R_{i_1}^{k_1}R_{i_2}^{k_2}\dots R_{i_n}^{k_n}=I$, then there exists two consecutive indices $i_j$, $i_{j+1}$ such that $i_j=i_{j+1}$.

Define the set of words $\mathcal{A}^*$ on alphabet $\mathcal{A}=\{a,a^{-1},b,b^{-1}\}$ and the set of admissible words $\mathcal{A}^*_0$ such that no two consecutive letters are inverse from each other. Map letter $a$ to $R_1$, $a^{-1}$ to $R_1^{-1}$, letter $b$ to $R_2$ and $b^{-1}$ to $R_2^{-1}$. The concatenation of words is mapped to the product of corresponding rotations. We refer to this mapping as $\varphi:\mathcal{A}_0^*\to SO_3$. For instance $\varphi(bab^{-1}a)=R_2R_1R_2^{-1}R_1$. Since, $R_1$ and $R_2$ are algebraically independent, map $\varphi$ is injective. Consider the directed Cayley graph $\mathcal{G}$ with vertices $\mathcal{V}=\mathcal{A}^*_0$ , edges set $E$ defined as:
\[(w,w')\in\mathcal{E}\Leftrightarrow\exists l\in\mathcal{A}, w'=wl\;.\]
Define the endpoint maps $\partial_0$ and $\partial_1$ $:E \to \mathcal{V}$ such that $\partial_0(e)=w$ and $\partial_1(e)=w'$ for $e=(w,w')\in E$. Equipped with its endpoint maps, $\mathcal{G}$ is called a \emph{combinatorial graph}. To turn $\mathcal{G}$ into a metric graph, we follow a standard construction, see, \emph{e.g.} \cite[p.7]{bridson:haefliger:1999}. Let us form the quotient set $X_{\mathcal{G}}=E\times [0,1]/\sim$ where the equivalence relation $\sim$ is such that $(e,i)\sim (e',i')$ iff $\partial_i(e)=\partial_{i'}(e')$, with $i,i'\in \{0,1\}$. We adopt the convention to choose $(e,1)$ to represent the equivalence class $\{(e,1),(e',0)\}$ when $\partial_0(e')=\partial_1(e)$. We then equip $X_{\mathcal{G}}$ with the standard metric distance $d_X$, as described, \emph{e.g.} in \cite[p.7]{bridson:haefliger:1999}.

To each couple $(e,\lambda)\in E \times [0,1]$ such that $e=(w,w')$, and $\partial_1(e)=x_1\dots x_n\in \mathcal{A}_0^*$ with  $x_i\in\mathcal{A}$ for all $i\leq n$ and $\lambda\in [0,1]$ we assign a rotation: $\psi(e,\lambda)=\varphi(x_1)\dots \varphi(x_{n-1})\varphi(x_n)^{\lambda}$. And denote by $\bar{\psi}$ its induced map on the quotient space $X_{\mathcal{G}}$. Because of the quotient identifications and the fact that $\varphi$ is injective, $\bar{\psi}$ is in turn injective on $X_{\mathcal{G}}$. The image $\mathcal{M}=\bar{\psi}(X_{\mathcal{G}})\subset SO_3$ is our state space, and its metric distance $d$ is either taken as the geodesic distance on $SO_3$ when geodesics are restricted to the set $\mathcal{M}$, or equivalently, derived from the distance on $X_{\mathcal{G}}$ by $d(x,y)=d_X(\bar{\psi}^{-1}(x),\bar{\psi}^{-1}(y))$ ($\bar\psi$ is an isometry from $(X_{\mathcal{G}},d_X)$ to $(\mathcal{M},d)$). The metric space thus defined is $CAT(0)$ \cite[p.167]{bridson:haefliger:1999}. Hence, in $(\mathcal{M},d)$, midpoints are well defined. For a simple illustration of this formal construction, see Figure~\ref{fig:metric graph}.

To compute the distance between two points $x_1\in\mathcal{M}$ and $x_2\in \mathcal{M}$. Let $(e_1,\lambda_1)=\bar{\psi}^{-1}(x_1)\in X_{\mathcal{G}}$ and $(e_2,\lambda_2)=\bar{\psi}^{-1}(x_2)\in X_{\mathcal{G}}$; $w_1=\partial_0(e_1) \in \mathcal{A}_0^*$ and $w_2=\partial_0(e_2)\in \mathcal{A}_0^*$. Denote by $p=s_1\dots s_p\in \mathcal{A}_0^*$ the longest common prefix of the words $w_1$ and $w_2$. Then we have two cases: Either $p\in\{w_1,w_2\}$, or $p\not\in\{w_1,w_2\}$. Suppose, for instance (the other case would be treated in the same way), that $p=w_2$, and $w_1 \neq w_2$: $w_2$ is a prefix of $w_1$. Then distance $d(x_1,x_2)$ is given by $|w_1|-|w_2|+\lambda_2-\lambda_1$ where $|w_1|$ (resp. $|w_2|$) is the length of the word $w_1$ (resp. $|w_2|$). If $w_1=w_2$ simply take $d(x_1,x_2)=|\lambda_2-\lambda_1|$. The second case is when $p\neq w_1$ and $p\neq w_2$ then, denoting $z=\bar{\psi}(e_p,1)$, where $e_p=(p_-,p)$ with $p_{-}=s_1\dots s_{p-1}$ and $p=s_1\dots s_p$, we get $d(x_1,x_2)=d(x_1,z)+d(z,x_2)=|w_1|+|w_2|-2|p|+\lambda_2+\lambda_1$.

To compute the midpoint $\langle \frac{x_1+x_2}{2} \rangle$ of two points $(x_1,x_2)\in \mathcal{M}^2$. Let $(e_1,\lambda_1)=\bar{\psi}^{-1}(x)$, $(e_2,\lambda_2)=\bar{\psi}^{-1}(y)$, $w_1=s_1\dots s_{l_1}=\partial_0(e_1)$ and $w_2=t_1\dots t_{l_2}=\partial_0(e_2)$. Let $p=u_1\dots u_{l_p}\in \mathcal{A}_0^*$ be the longest common prefix of $w_1$ and $w_2$ and $x_p=\bar{\psi}(e_p,1)$ ($e_p=(p_{-},p)$ is defined the same way as in the paragraph above). Here we have: $s_i,t_i,u_i\in \mathcal{A}$ and $l_1,l_2,l_p \geq 0$. Denote $D=\frac{d(x_1,z)+d(z,x_2)}{2}$ and $\langle \frac{x_1+x_2}{2} \rangle=(m,\lambda_{m})$. We have two cases:
\begin{itemize}
\item If $D < d(x_1,z)$ then: $m=u_1\dots u_{l_p}t_{l_p+1}\dots t_{L+1}$ and $\lambda_{m} =d(x_1,z)-D-L $; where $L=\lfloor d(x_1,z)-D \rfloor$.

\item If $D > d(x_1,z)$ then : $m=u_{1}u_2\dots u_{l_p}s_{l_p+1}\dots s_{L+1}$ where $L=\lfloor D-d(x_1,z) \rfloor$. And $\lambda_{m} =D-d(x,z)-L$.
\end{itemize}
We simulate the algorithm on a network of $N=20$ agents using elements $x_1,\dots,x_N\in\mathcal{M}$ such that for all $1\leq i\leq N$, $x_i=\bar\psi(e_i,\lambda_i)$ where $e_i=(w_{i,1},w_{i,2})$; the $\{w_{i,j}\}_{i\leq N, j\in \{1,2\}}\in \mathcal{A}_0^*$ are words of length $l_i\leq 30$. To generate the sequence $(e_i)_{1 \leq i \leq N}$, it suffices to generate the sequence $(w_{i,2})_{1 \leq i \leq N}$ since $w_{i,1}$ can be deduced from $w_{i,2}$ by removing its last letter. In order to generate an element $w_{i_0,2}$ of $\{w_{i,2}\}_{1\leq i \leq N}$, first we need to specify its length $l_{i_0}$ which is a positive integer randomly and uniformly chosen from $\{1,\dots,30\}$. We then construct $w_{i_0,2}$ in the following way: start by randomly and uniformly choosing a letter $s_1\in\mathcal{A}$, then for $k\in \{2,\dots,l_{i_0}\}$ randomly and uniformly choose a letter $s_k\in\mathcal{A}$; if $s_k=s_{k-1}^{-1}$ then re-sample $s_k$ again until $s_{k}\neq s_{k-1}^{-1}$. After the sequence $\{w_{i,j}\}_{i\leq N, j\in \{1,2\}}$ is generated, the  $(\lambda_i)_{1\leq i\leq N}$ are sampled uniformly in $[0,1]$. The underlying network is the complete graph $K_N$. In figure~\ref{fig:F2} we plot $n \to \log\sigma^2_n$ and obtain a result that is in accordance with the prediction of theorem~\ref{the:speed}: one can observe a linearly decreasing $\log\sigma^2_n$ (or, equivalently $\log\Delta_n$) which means that consensus is indeed achieved exponentially fast.
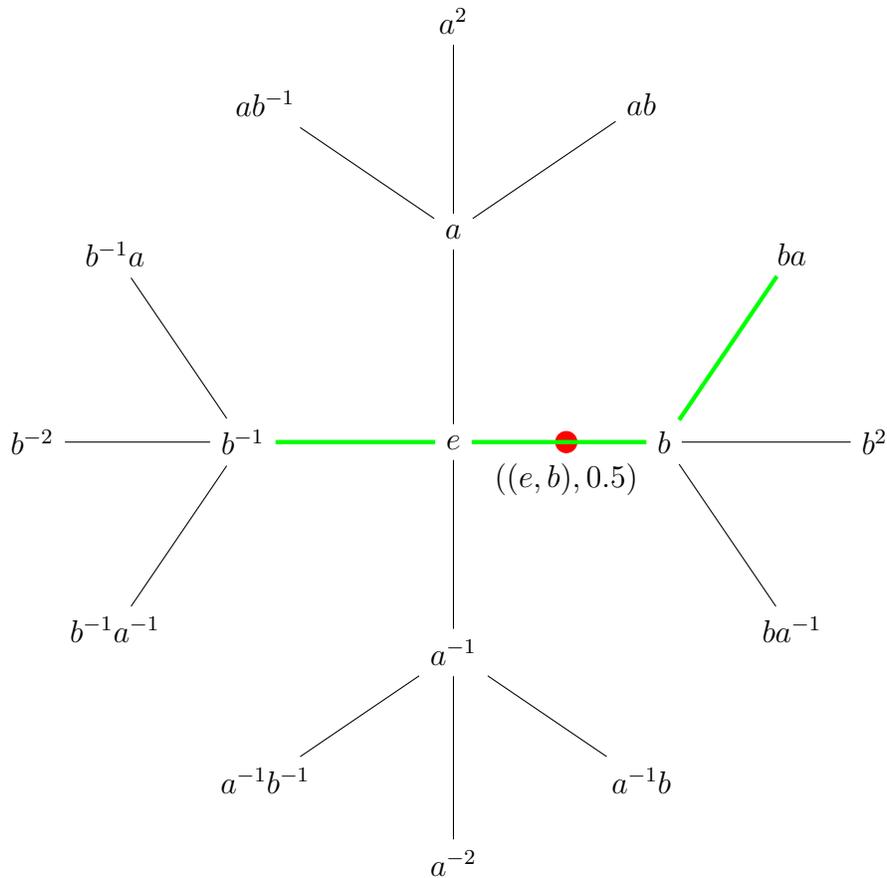
\begin{figure}[ht!]
\centering
\begin{tikzpicture}
\node (e) at (0,0) {$e$};
\node (a) at (2.8,0) {$b$};
\node (-a) at (-2.8,0) {$b^{-1}$};
\node (b) at (0,2.8) {$a$};
\node (-b) at (0,-2.8) {$a^{-1}$};
\node (2a) at (5.6,0) {$b^2$};
\node (2b) at (0,5.6) {$a^2$};
\node (a-b) at (4.5,-2.5) {$ba^{-1}$};
\node (-2b) at (0,-5.6) {$a^{-2}$};
\node (-2a) at (-5.6,0) {$b^{-2}$};
\node (-b+a) at (2.5,-4.5) {$a^{-1}b$};
\node (-b-a) at (-2.5,-4.5) {$a^{-1}b^{-1}$};
\node (a+b) at (4.5,2.5) {$ba$};
\node (b+a) at (2.5,4.5) {$ab$};
\node (b-a) at (-2.5,4.5) {$ab^{-1}$};
\node (-a+b) at (-4.5,2.5) {$b^{-1}a$};
\node (-a-b) at (-4.5,-2.5) {$b^{-1}a^{-1}$};
\filldraw[red] (1.5,0) circle (4pt) node[anchor=south] {};
\node (m') at (1.5,-0.5) {$((e,b),0.5)$};

\draw [ultra thick, green] (e)--(a);
\draw [ultra thick, green] (e)--(-a);
\draw (e)--(b);
\draw (e)--(-b);
\draw (a)--(2a);
\draw (-a)--(-2a);
\draw (b)--(2b);
\draw (-b)--(-2b);
\draw (a)--(a-b);
\draw (-b)--(-b+a);
\draw (-b)--(-b-a);
\draw [ultra thick, green] (a)--(a+b);
\draw (b)--(b+a);
\draw (b)--(b-a);
\draw (-a)--(-a+b);
\draw (-a)--(-a-b);
\end{tikzpicture}
\caption{Metric graph of words of length $2$ from the alphabet $\mathcal{A}$. All edges have the same length $1$. To draw a point $x=(e,\lambda)$ on the graph, first go the vertex $e$. Then, seen as a copy of the segment $(0,1]$, one can draw a point on the segment such that its distance from $\partial_0(e)$ is $\lambda$. After drawing two points $x_1$ and $x_2$, it becomes easy to find the shortest path between them. In this example, $x_1=((e,b^{-1}),1)$ and $x_2=((b,ba),1)$ the geodesic relating them is drawn in green, and their midpoint $\langle\frac{x_1+x_2}{2}\rangle=((e,b),0.5)$ is seen in red. $d(x_1,x_2)=3$.}
\label{fig:metric graph}
\end{figure}

\begin{figure}[ht!]
  \centering
  \includegraphics[width=0.8\linewidth]{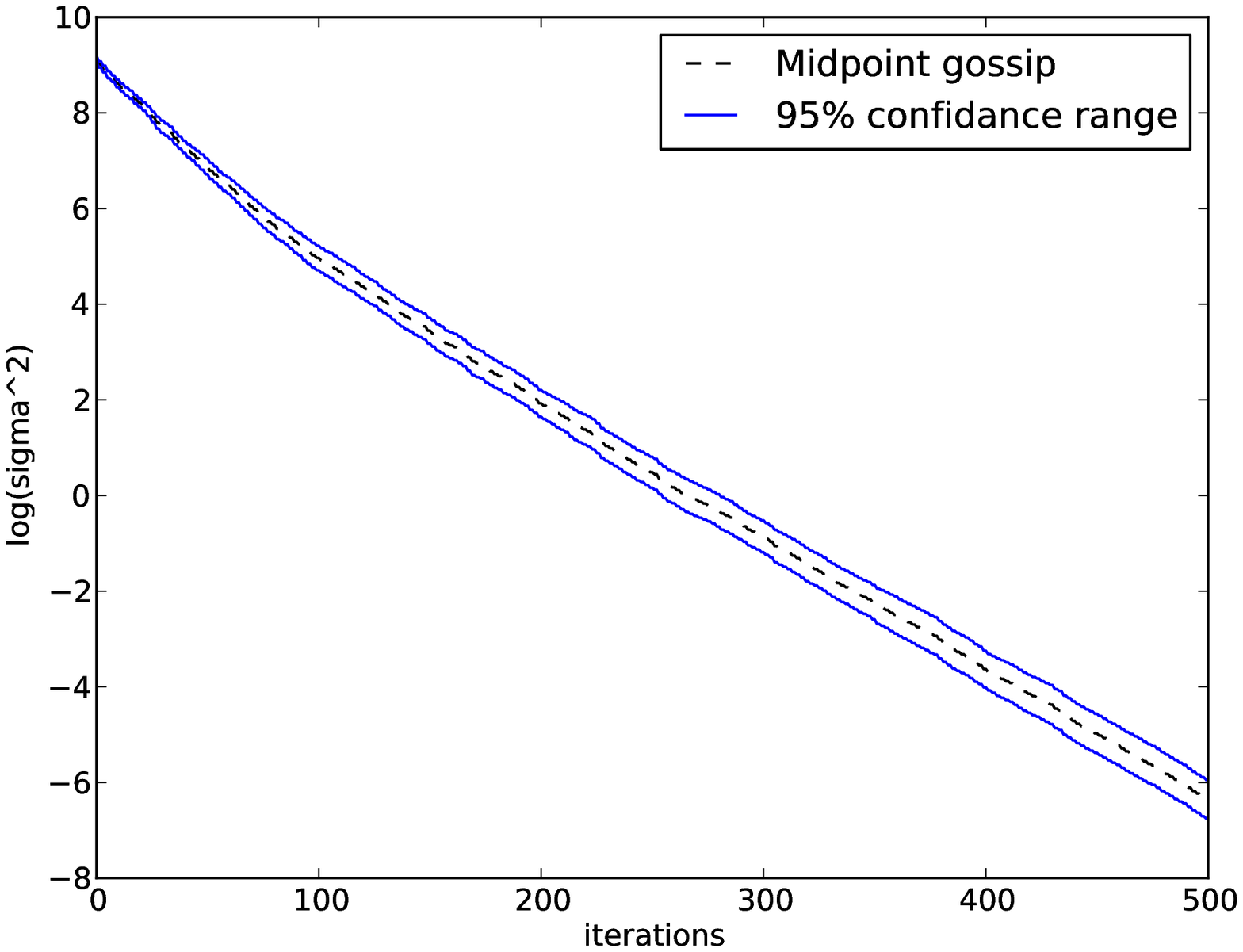}
  \caption{Plot of $n\mapsto \log\sigma^2_n$ for the metric graph. Because of the stochastic nature of the algorithm, $50$ simulations are done and we plot $\log\sigma_n^2$ as well as a confidence domain which contains $95\%$ of the simulated curves.}
  \label{fig:F2}
\end{figure}

\subsection{The sphere}

In this example, we shall consider the $3$-dimensional unit sphere $S^2=\{(x,y,z)\in \mathbb{R}^3| x^2+y^2+z^2=1\}$ equipped with the distance $d(a,b)=\cos^{-1}(\langle a,b \rangle)$ such that $0\leq d(a,b)\leq \pi$ for all $a,b \in S^2$. Since, two antipodal points on the sphere have an infinite number of minimizing geodesics linking them, we sample the initial set of points inside a small portion of the sphere. Quantitatively, we choose the quarter of a sphere; in our numerical experiments we chose $\mathcal{Q}=\{(x,y,z) \in S^2| x >0, y >0, z >0\}$ which is of diameter $r_1=\frac{\pi}{2}$ thus convex and thus $CAT(1)$ (as a convex subset of the model space $\mathcal{M}_1^3$ -- with an abuse of language since $\mathcal{M}_\kappa^n$ is only defined up to an isometry).

Note that the sphere does not possess a vector space structure and thus one cannot use classical Arithmetic gossiping without a reprojection step.

We sample a set of $N=30$ points uniformly from $\mathcal{Q}$ as initial step.
The expression of a geodesic $\gamma(t)$ on $\mathcal{Q}$ such that $\gamma(0)=p$ and $\gamma(1)=q$ and $p \neq q$ is given by:
\[\gamma(t)=\sin\bigg(\cos^{-1}(\langle p,q\rangle)t\bigg) \frac{q-\langle p,q \rangle p}{\sqrt{1-\langle p,q \rangle^2}}+\cos\bigg(\cos^{-1}(\langle p,q\rangle)t\bigg)p\;.\]

The total number of iterations is 500, for the graph, we use a complete graph and the path graph. By plotting the variance function $\sigma^2_1(X_n)=\frac{1}{N}\sum_{\{i,j\}\in\mathcal{P}_2(V)}\chi_1(d(x_i(n),x_j(n))$ with respect to the number of iterations we get figure~\ref{fig:Sphere} (for the complete graph) and figure~\ref{fig:SpherePn} (for the path graph): we observe in both cases that $n\mapsto \log\sigma^2_1(X_n)$ (or equivalently $n\mapsto \log\Delta_1$) is a linear function with negative slope which is in accordance with theorem~\ref{the:speedk}. Convergence in the case of the path graph is slower than that of the complete graph (the slope of $n\mapsto \log\sigma^2_1(X_n)$ for the path graph is smaller in absolute value than the one for the complete graph), which highlights the influence of graph connectivity on the speed of convergence.

\begin{figure}[ht!]
  \centering
  \includegraphics[width=0.8\linewidth]{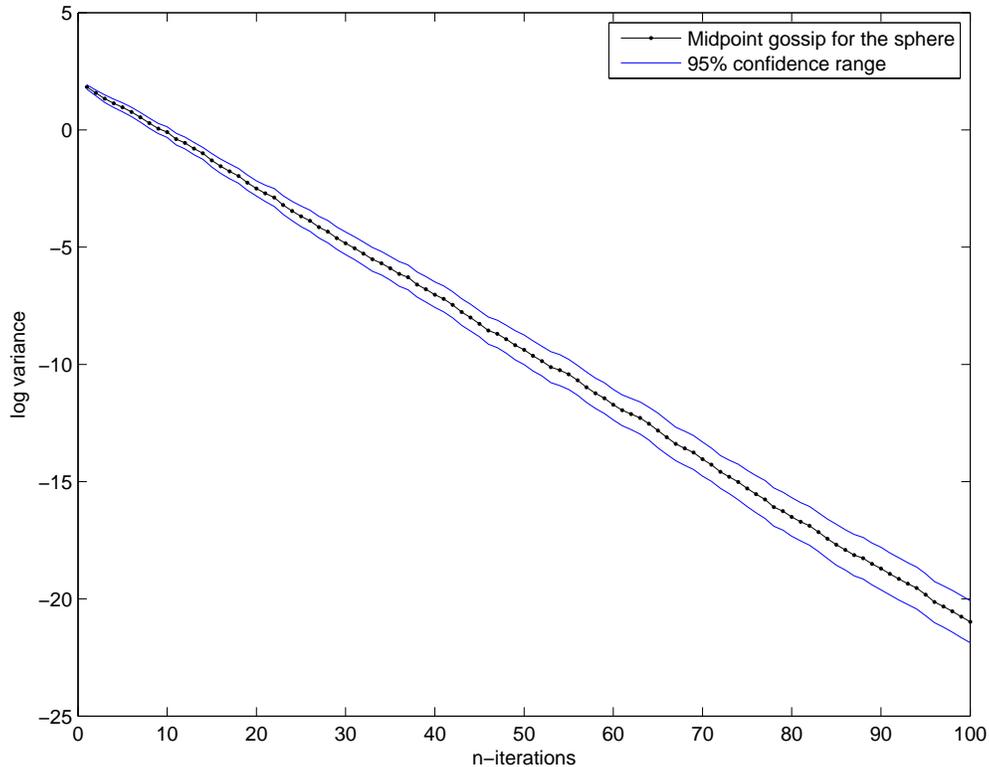}
  \caption{Plot of $n\mapsto \log\sigma^2_1(X_n)$ where $n$ is the iteration index and $\mathcal{M}=S^2$ (complete graph $K_N$). Because of the stochastic nature of the algorithm, $100$ simulations are done and we plot $\log\sigma^2_1$ as well as a confidence domain which contains $95\%$ of the simulated curves.}
  \label{fig:Sphere}
\end{figure}

\begin{figure}[ht!]
  \centering
  \includegraphics[width=0.8\linewidth]{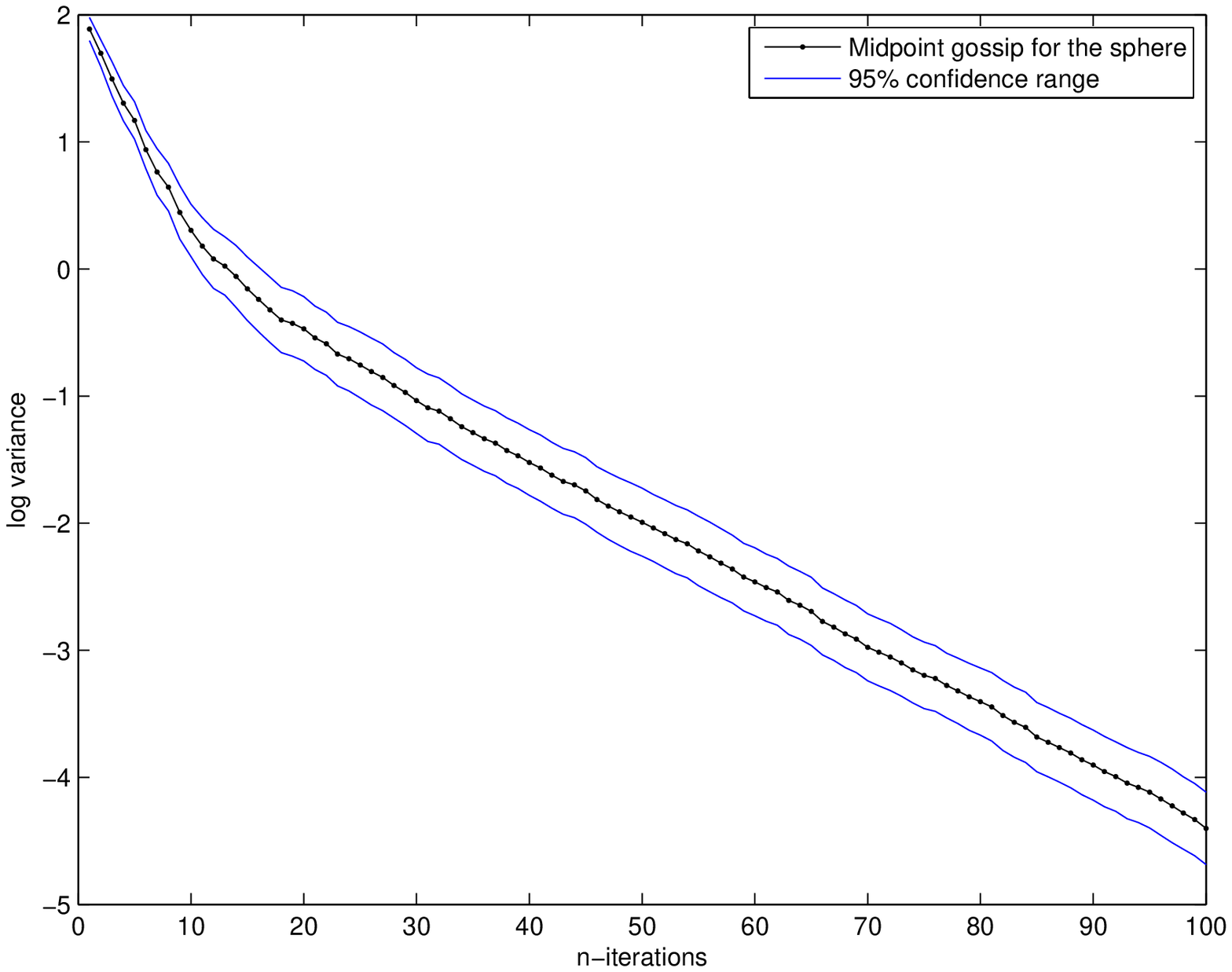}
  \caption{Plot of $n\mapsto \log\sigma^2_1(X_n)$ where $n$ is the iteration index and $\mathcal{M}=S^2$ (path graph $P_n$). The graph is complete. Because of the stochastic nature of the algorithm, $50$ simulations are done and we plot $\log\sigma^2_1$ as well as a confidence domain which contains $95\%$ of the simulated curves here again the slope is smaller then the complete graph case.}
  \label{fig:SpherePn}
\end{figure}

\subsection{Group of rotations}

We shall be interested in what follows in the \emph{rotations group} $SO_3$ of the Euclidean space $\mathbb{R}^3$. A rotation $R_{a,\theta}$ acting on $\mathbb{R}^3$ is characterized by its axis of rotation $a\in \mathbb{R}^3$ and its rotation angle $\theta \in [-\pi,\pi)$; the eigenvalues of $R_{a,\theta}$ are: $\{e^{i\theta},e^{-i\theta},1\}$. One of the possible applications of Algorithm~\ref{algo:gpapa} with data in $SO_3$ is a network of $3D$ cameras \cite{tron2008distributed} that seeks to achieve a consensus in order to estimate the pose of an object.

$SO_3$ is a Lie group, its Lie algebra is $\mathfrak{so}_3$ the space of $3\times 3$ skew-symmetric matrices. The Riemannian metric at identity is given by $\langle v_1, v_2 \rangle= \frac{1}{2}\tr(v_1^Tv_2)$ for $v_1,v_2 \in \mathfrak{so}_3$. In this metric, the sectional curvature of $SO_3$ is given by the formula: $\kappa(\sigma)=\frac{1}{4}||[X,Y]||^2$ with $X,Y\in\mathfrak{so}_3$ orthonormal generators of $\sigma$ \cite[p.103]{DoCarmo}. The collection of matrices $(M_{i,j})_{1 \leq i < j \leq 3}$  such that: $M_{i,j}=(E_{i,j}-E_{j,i})$ forms an orthonormal basis for the space of skew-symmetric matrices, where $(E_{i,j})_{1 \leq i,j\leq 3}$ is the canonical basis of $M_3(\mathbb{R})$. One can check on this basis that $k(\sigma)\equiv\frac14$. The sectional curvature of $SO_3$ is thus identically $\kappa\equiv\frac{1}{4}$. This implies that $r_\kappa=\frac{\pi}{2}$. Toponogov comparison theorem \cite[p.400]{Chavel} shows that the geodesic ball $\mathcal{B}$ with center $I_3$ and diameter $r_{\kappa}$ is a $CAT(\frac{1}{4})$ space.

The exponential function $\exp:\mathfrak{so}_3 \mapsto SO_3$ is given by the convergent series $\exp(X)=\sum_{k=0}^{\infty} \frac{X^k}{k!}$. Since the injectivity radius of $SO_3$ is $>\pi$~{\cite[p.406]{Chavel}, $\exp$ is a diffeomorphism from $\mathfrak{B}(0,r_\kappa) \subset \mathfrak{so}_3$ to $\mathcal{B}$. If $R_{a,\theta}$ is a rotation matrix, we say that $X\in \mathfrak{so}_3$ is a logarithm of $R$ iff: $\exp(X)=R_{a,\theta}$. When $R_{a,\theta}$ does not have $-1$ as an eigenvalue, it is possible to define the \emph{principal logarithm} $\log(R_{a, \theta})$ such that the eigenvalues of $\log(R_{a, \theta})$ lie in $\mathcal{S}=\{z\in \mathbb{C}| -\pi <\Im(z)<\pi \}$. For example, the matrix: \[X_k=(\theta+2\pi k)\left[ {\begin{array}{ccc} 0 & -1 & 0\\ 1 & 0 &0\\ 0 & 0 &0\\ \end{array}} \right]\] with $\theta \in (-\pi,\pi)$, is a logarithm of: \[R_{z,\theta}=\left[ {\begin{array}{ccc} \cos(\theta) & -\sin(\theta) & 0\\ \sin(\theta) & \cos(\theta) & 0\\ 0 & 0 & 1\\ \end{array}} \right] \] for every integer $k\in\mathbb{Z}$ but the principal logarithm of $R_{z,\theta}$ is $\log(R_{z,\theta})=X_0$. In what follows $\log$ will denote the principal logarithm function.

Let, $(R_{a_1,\theta_1},R_{a_2,\theta_2}) \in \mathcal{B}^2$. If $-1$ is not an eigenvalue of $R_{a_1,\theta_1}^TR_{a_2,\theta_2}$, then the distance between the two elements is: \[d(R_{a_1,\theta_1},R_{a_2,\theta_2})^2=\frac{1}{2}||\log(R_{a_1,\theta_1}^TR_{a_2,\theta_2})||^2=[\alpha]^2\] where $\{e^{i[\alpha]},e^{-i[\alpha]},1\}$ are the eigenvalues of $(R_{a_1,\theta_1}^TR_{a_2,\theta_2})$, such that $[\alpha] \in (-\pi,\pi)$. If ${-1}$ is an eigenvalue of $R_{a_1,\theta_1}^TR_{a_2,\theta_2}$ then $d(R_{a_1,\theta_1},R_{a_2,\theta_2})=\pi$, and ($R_{a_1,\theta_1}$, $R_{a_2,\theta_2}$) are said to be \emph{antipodal points}.

For $R_{a,\theta}\in \mathcal{B}$ we have $d(I_3,R_{a,\theta})=|\theta|<\frac{\pi}{4}$, and for $(R_{a_1,\theta_1},R_{a_2,\theta_2}) \in \mathcal{B}^2$ we have $|[\alpha ]|=d(R_{a_1,\theta_1},R_{a_2,\theta_2})\leq d(I_3,R_{a_1,\theta_1})+d(I_3,R_{a_2,\theta_2})<\frac{\pi}{2}$ which implies that $-1$ cannot be an eigenvalue of $R_{a_1,\theta_1}^TR_{a_2,\theta_2}$. Thus $\mathcal{B}$ does not contain antipodal points.

For all $(R_{a_1,\theta_1},R_{a_2,\theta_2})\in \mathcal{B}^2$ there exists a unique minimizing geodesic $\gamma(t)$ such that $\gamma(0)=R_{a_1, \theta_1}$ and $\gamma(1)=R_{a_2, \theta_2}$, and it has the following expression: \[\gamma(t)=R_{a_1,\theta_1}\exp\left(t\log(R_{a_1,\theta_1}^TR_{a_2,\theta_2})\right)\]
Since $\mathcal{B}$ is strongly convex \cite[p.404]{Chavel}, $\gamma(t)\in \mathcal{B}$ for all $t\in[0,1]$. The expression of the midpoint is thus: $\langle \frac{R_{a_1,\theta_1}+R_{a_2,\theta_2}}{2} \rangle = \sqrt{R_{a_1,\theta_1}R_{a_2,\theta_2}}$.

In the numerical simulation presented in this paper, we sample $N=30$ rotation matrices $(R_i)_{1\leq i\leq N}\in\mathcal{B}$. The underlying graph is the complete graph $K_N$. The results of the experiment are displayed in figure~\ref{fig:ortho} where we plot the logarithm of: $\sigma^2_{\frac{1}{4}}(X_n)$ as a function of $n$. We observe that $n\mapsto\log \sigma^2_{\frac{1}{4}}(X_n)$ decreases linearly, which is in accordance with theorem~\ref{the:speedk}.

\begin{figure}[ht!]
  \centering
  \includegraphics[width=0.8\linewidth]{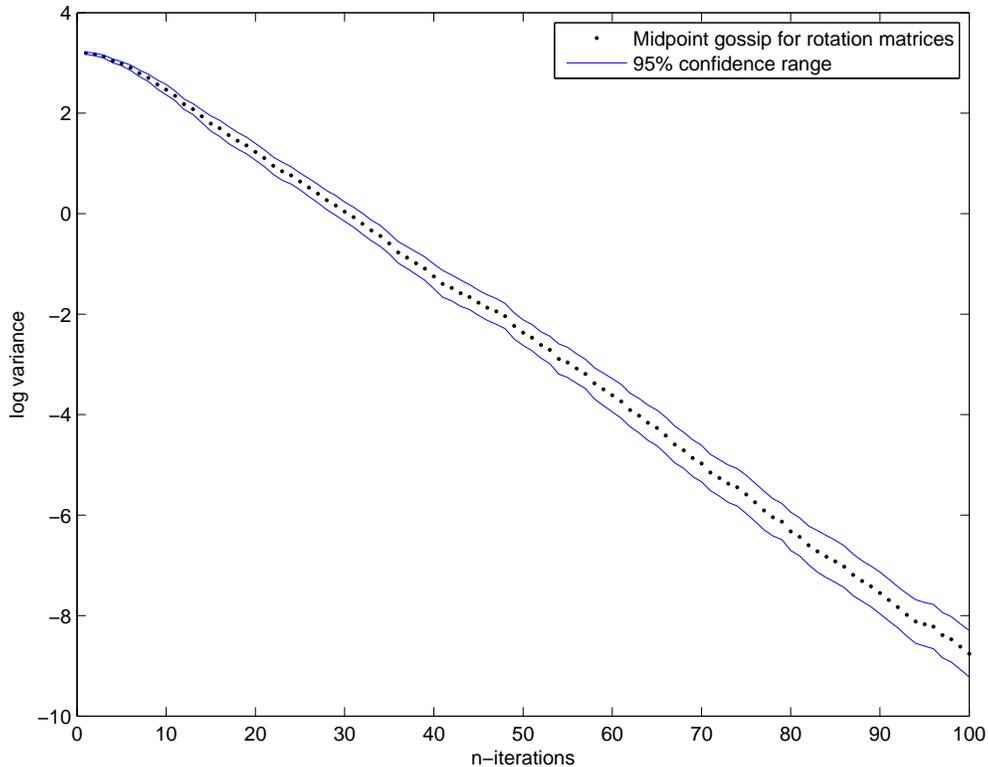}
  \caption{Plot of $n\mapsto\log \sigma^2_{\frac{1}{4}}(n)$ where $n$ is the iteration index, and $\mathcal{M}=SO_3$. Because of the stochastic nature of the algorithm, $50$ simulations are done and we plot $\log\sigma^2_{\frac{1}{4}}$ as well as a confidence domain which contains $95\%$ of the simulated curves. The average midpoint algorithm exhibits exponential convergence towards a consensus state.}
  \label{fig:ortho}
\end{figure}

\section{Acknowledgments}

The authors warmly thanks Julien Hendrickx for fruitful discussions.

\section{Conclusion}
\label{sec:conclu}
We have presented an extension to the (RPG) to the case of $CAT(\kappa)$ spaces in the asynchronous pairwise case. We identified a set of conditions on the curvature ($\kappa=0$) that guarantees a global convergence of the Random Pairwise Midpoint, for $\kappa > 0 $ only a local convergence result has been proven. The algorithm converges in each case towards an arbitrary consensus state at exponential speed. Our experiments with positive definite matrices, the metric graph associated to the free group with two generators, the sphere, and the three dimensional special orthogonal group agree with theoretical results and validate our approach.


\appendix

\section{$CAT(\kappa)$ metric spaces}
\label{sec:catk}

Unless further qualification is made, in all that follows $\kappa$ will denote an arbitrary real number and $n$ an integer strictly greater than $1$. $CAT(\kappa)$ metric spaces are defined using comparisons with model spaces that are defined below. Assuming familiarity with Riemannian Geometry, the \emph{model space} $\mathcal{M}_\kappa^n$ denotes any complete, simply connected, $n$-dimensional Riemannian manifold, with constant sectional curvature $\kappa$. It can be shown that all such Riemannian manifolds are indeed isometric, hence the name ``the'' model space $\mathcal{M}_\kappa^n$. However, for the sake of completeness and readability, we follow the treatment of, \emph{e.g.} \cite{bridson:haefliger:1999}, and provide below a simple metric construction of $\mathcal{M}_\kappa^n$, freed from any reference to differential geometry.

\subsection{Model Space $\mathcal{M}_\kappa^n$}

In order to properly define the model space, we need three prototype spaces: the euclidean space, the sphere and the hyperbolic space. General model spaces are then simply derived by dilation.

Let us denote $\mathcal{E}^{n}$ the vector space $\mathbb{R}^{n}$ equipped with its standard Euclidean norm $\|x\|^2 = \sum_{i=0}^{n-1} x_i^2$ with $x=(x_0,\dots, x_{n-1})\in\mathcal{E}^{n}$. Denote
\[
\mathcal{S}^n = \{(x_0,\dots,x_n)\in\mathcal{E}^{n+1}: \sum_{i=0}^n x_i^2 = 1\}\,,
\]
the $n$-dimensional unit Euclidean sphere and
\[
\mathcal{H}^n = \{(x_0,\dots,x_n)\in\mathcal{E}^{n+1}:(-x_0^2+\sum_{i=1}^nx_i^2) = -1, x_0>0\}\,,
\]
one sheet of a two-sheets $n$-dimensional hyperboloid.
As a metric space $\mathcal{E}^{n}$ is equipped with distance $d_E(x,y) = \|x-y\|^2$, $\mathcal{S}_\kappa^n$ is equipped with distance $0\leq d_S(x,y)\leq \pi$ such that
\[
\cos\left(d_S(x,y)\right) = \sum_{i=0}^n x_iy_i\,,
\]
whereas $\mathcal{H}_n$ is equipped with distance $d_H(x,y)\geq 0$ such that:
\[
\cosh \left(d_H(x,y)\right) = x_0y_0 - \sum_{i=1}^nx_iy_i\;.
\]
\begin{rem}
Function $d_S$ is well defined since for $(x,y)\in(\mathcal{S}^n)^2$, $-1\leq \sum x_iy_i\leq 1$. Note that if $(x_0,\dots,x_n)\in\mathcal{H}^n$ then necessarily $x_0\geq 1$. For all $(x,y)\in\mathcal{H}^n$,
\[
\sum_{i=1}^nx_iy_i\leq\left(\sum_{i=1}^nx_i^2\right)^{1/2}\left(\sum_{i=1}^ny_i^2 \right)^{1/2}= (x_0^2 - 1)^{1/2}(y_0^2 - 1)^{1/2}\leq x_0y_0\;.
\]
Thus, function $d_H$ is well defined. We refer to \cite[chap~I.2]{bridson:haefliger:1999} for the proof that $d_S$ and $d_H$ satisfy the requirements of distance functions (triangle inequality is not obvious).
\end{rem}
\begin{rem}\label{rem:diam}
The diameter $\diam(\mathcal{S}^n)$ equals $\pi$ and is attained for any two \emph{antipodal} points $x$ and $y=-x$, while $\sup_{(x,y)\in(\mathcal{H}^n)^2}d(x,y) = +\infty$.
\end{rem}
We are now in a position to provide a definition of the model spaces $\mathcal{M}_\kappa^n$.
\begin{defn}
The \emph{model space} $(\mathcal{M}_\kappa^n,\bar d)$ is the metric space defined by:
\[
\begin{cases}
  \mathcal{M}_\kappa^n = \mathcal{H}^n, \bar d=|\kappa|^{-1/2} d_H(\cdot,\cdot) & \text{if }\kappa < 0\\
  \mathcal{M}_\kappa^n = \mathcal{E}^n, \bar d = d_E(\cdot,\cdot) & \text{if }\kappa = 0\\
  \mathcal{M}_\kappa^n = \mathcal{S}^n, \bar d = \kappa^{-1/2} d_S(\cdot,\cdot) & \text{if }\kappa > 0
\end{cases}
\]
\end{defn}
\begin{rem}
We only make use of $n=2$ in the defining equality of $CAT(\kappa)$ spaces. However, we would not have gained much restricting ourselves to the case $n=2$ to define the previous model spaces.
\end{rem}
The following proposition is easily derived from remark~\ref{rem:diam}.
\begin{prop}
The diameter of $\mathcal{M}_\kappa^n$ is given by $D_\kappa$, where:
\[
D_\kappa =
\begin{cases}
+\infty & \text{if }\kappa\leq 0\\
\frac{\pi}{\sqrt\kappa} & \text{if }\kappa>0
\end{cases}
\]
\end{prop}
In what follows we also use notation $r_\kappa = \frac{D_\kappa}2$.

\subsection{Segments, Length, Angle, Triangles}

Let us recall some definitions related to metric spaces. Details can be found, \emph{e.g.} in \cite{bridson:haefliger:1999}.

The following definition generalizes the Euclidean case, where $\|x-y\| = \|x-z\| + \|z-y\|$ implies that $z$ belongs to the segment $[x,y]$.
Throughout the rest of the paper $(\mathcal{M},d)$ denotes a metric space.
\begin{defn}[Geodesic, Segments]
A path $c:[0,l]\to\mathcal{M}$, $l\geq0$ is said to be a \emph{geodesic} if $d(c(t),c(t'))=|t-t'|$, for all $(t,t')\in[0,l]^2$; $x=c(0)$ and $y=c(1)$ are the \emph{endpoints} of the geodesic and $l = d(x,y)$ is the \emph{length} of the geodesic. The image of $c$ is called a \emph{geodesic segment} with endpoints $x$ and $y$. If there is a single segment with endpoints $x$ and $y$, it is denoted $[x,y]$.
\end{defn}

\begin{defn}[Midpoint]
The \emph{midpoint} of segment $[x,y]$ is denoted $\left\langle \frac{x+y}2\right\rangle$, it is defined as the unique point $m$ such that $d(x,m) = d(y,m) = d(x,y)/2$.
\end{defn}
Please note that defining $\left\langle\frac{x+y}2\right\rangle$ involves actually no addition nor scalar multiplication. This notation makes an analogy of the Euclidean case where midpoints indeed correspond to arithmetic means.

\begin{defn}[Triangle]
A triplet of geodesic $c_i:[0,l_i]\to\mathcal{M}$ with $i=0,1,2$ is said a geodesic triangle if and only if $c_i(l_i) = c_{i+1\mod 3}(0)$. Images of the geodesics $c_i$ are called the \emph{sides} of the triangle, their endpoints $c_i(0)$ are called the \emph{vertices} of the triangle.
\end{defn}

\begin{defn}[Comparison Triangles in $\mathcal{M}_\kappa^n$]
Assume $p,q$ and $r$ are three points in $\mathcal{M}$. A comparison triangle in $\mathcal{M}_\kappa^n$ refers to any three points, provided they exist, $\bar p,\bar q$ and $\bar r$ in $\mathcal{M}_\kappa^n$ such that $\bar d(\bar p,\bar q) = d(p,q)$, $\bar d(\bar q,\bar r) = d(q,r)$, and $\bar d(\bar r,\bar p) = d(r,p)$.
\end{defn}
Concerning the existence and uniqueness of such comparison triangles, we provide without proof the following proposition (see for instance \cite{bridson:haefliger:1999} for a proof).
\begin{prop}
Assume $p,q$ and $r$ are three points in $\mathcal{M}$ such that $d(p,q)+d(q,r)+d(r,p)<2D_\kappa$ and $max(d(p,q),d(q,r),d(r,p))\leq D_\kappa$. Then there exists a comparison triangle in $\mathcal{M}_\kappa^n$. Moreover, this comparison triangle is unique up to an isometry.
\end{prop}
\begin{rem}
Note that the proposition is straightforward for $\kappa=0$, where triangle inequality is a necessary and sufficient condition for existence of comparison triangles and is automatically satisfied for a triplet of points in $\mathcal{M}$.
\end{rem}
There is also a notion of angle in this ``metric'' context, as illustrated by the next definition.
\begin{defn}[Alexandrov Angle]
Assume $c:[0,l]\to\mathcal{M}$ and $c':[0,l']\to\mathcal{M}$ are two geodesics such that $x = c(0)=c'(0)$, $y=c(l)$ and $z=c'(l')$. For each $0\leq t\leq l$ and $0\leq t'\leq l'$, consider a comparison triangle $(\bar x, \bar y_t,\bar z_{t'})$ in $\mathcal{M}_0^2$ for the triplet $(x,c(t),c'(t'))$. The angle between $c$ and $c'$ at $x$, denoted $\angle(c,c')$ is defined by:
\[
\angle(c,c') = \lim_{\epsilon\to 0}\sup_{0\leq t,t'\leq\epsilon}\bar\angle_{\bar x}([\bar x,\bar y_t],[\bar x,\bar z_{t'}])
\]
where $\bar\angle_{\bar x}([\bar x,\bar y_t],[\bar x,\bar z_{t'}])$ denotes the angle at $\bar x$ of the comparison triangle $(\bar x,\bar y_t,\bar z_{t'})$.
\end{defn}

\subsection{$CAT(\kappa)$ metric spaces}

The concepts from metric triangle geometry presented in the previous subsections yield the following definition of $CAT(\kappa)$ spaces.

\begin{defn}[$CAT(\kappa)$ inequality]
Assume $(\mathcal{M},d)$ is a metric space and $\Delta = (c_0,c_1,c_2)$ is a geodesic triangle with vertices $p = c_0(0)$, $q=c_1(0)$ and $r=c_2(0)$ and with perimeter strictly less than $2D_\kappa$. Let $\bar\Delta=(\bar p,\bar q,\bar r)$ denote a comparison triangle in $\mathcal{M}_\kappa^2$. $\Delta$ is said to satisfy the $CAT(\kappa)$ inequality if for any $x = c_0(t)$ and $y=c_2(t')$, one has:
\[
d(x,y) \leq \bar d(\bar x,\bar y)
\]
where $\bar x$ is the unique point of $[\bar p,\bar q]$ such that $d(p,x) = \bar d(\bar p,\bar q)$ and $\bar y$ on $[\bar p,\bar r]$ such that $d(p,y) = \bar d(\bar p,\bar y)$.
\end{defn}
\begin{rem}
Applying this inequality to the case where $\kappa\leq 0$ and $d(q,r)=0$, the uniqueness of geodesics is recovered when $\kappa\leq 0$.
\end{rem}

\begin{defn}[$CAT(\kappa)$ metric space]
A metric space $(\mathcal{M},d)$ is said $CAT(\kappa)$ if every geodesic triangle with perimeter less than $2D_\kappa$ satisfies the $CAT(\kappa)$ inequality.
\end{defn}
\begin{prop}\label{prop:cat0midpoints} If $x$ and $y$ are two points in a $CAT(0)$ space; there is a unique geodesic $[x,y]$ and the midpoint $\left\langle\frac{x+y}2\right\rangle$ is always well defined and unique.
\end{prop}

One has the so-called Bruhat-Tits inequality, which is a straightforward application of the $CAT(0)$ inequality:
\begin{prop}
Assume $(\mathcal{M},d)$ is $CAT(0)$, and $\Delta$ is a geodesic triangle with vertices $(p,q,r)$ such that $m$ is the midpoint of $q$ and $r$ along the side of the triangle; then,
\begin{equation}
2d(p,m)^2 \leq d(p,q)^2 + d(p,r)^2 - d(q,r)^2/2
\end{equation}
\end{prop}
\begin{rem}
Note that $CAT(0)$ inequality does apply to all geodesic triangles since the diameter restriction is void when $\kappa\leq 0$.
\end{rem}

For notational convenience, define the functions: $C_{\kappa}(t)=\cos(\sqrt{\kappa}t)$, $S_{\kappa}(t)=\frac{\sin(\sqrt{\kappa}t)}{\sqrt{\kappa}}$, $\chi_{\kappa}(t)=1-C_{\kappa}(t)$.

\begin{lem}[Law of Cosines]
\label{lem:cosines}
Given a complete manifold $\mathcal{M}_\kappa^n $ with constant sectional curvature $\kappa$ and a geodesic triangle $\Delta(pqr)$ in $\mathcal{M}_\kappa^n$, assume $\max\{d(p,r),d(q,r),d(p,q)\}<r$ and let $\alpha := \arc{prq}$. We have:

\[C_\kappa(d(p,q))= C_\kappa(d(p,r))C_\kappa(d(q,r))+S_\kappa(d(p,r))S_\kappa(d(q,r))\cos(\alpha) \]
\end{lem}

We deduce the following result.
\begin{lem}
\label{lem:trigo}
For for any triangle $\Delta(pqr)$ in $\mathcal{M}$ where $m$ is the midpoint of $[p,q]$ we have:
\[C_\kappa(d(p,r))+C_\kappa(d(q,r)) \leq 2C_\kappa(d(m,r)) C_\kappa \left(\frac{d(p,q)}{2}\right)\]
\end{lem}

\begin{proof}
Lets consider the triangle $\Delta(pqr)$ in  $\mathcal{M}$. We denote the geodesic midpoint of $p$ and $q$ by $m=\langle\frac{p+q}{2}\rangle$. Let $\Delta(p'q'r')$ be a comparison triangle to $\Delta(pqr)$ in $\mathcal{M}_\kappa^n$ and $m'$ a comparison point to $m$. A fundamental characterization of $CAT(\kappa)$~\cite{} spaces is that $d(r,m) < d(r',m')$. We apply lemma \ref{lem:trigo} to triangles comparison triangles $\Delta(p'm'r')$ and $\Delta(r'm'q')$.

\begin{figure}[ht!]
\centering
\begin{tikzpicture}

\draw (0,0)--(1,-2);
\draw (0,0)--(1,3);
\draw (1,3)--(1,-2);
\draw (0,0)--(1,0.5);
\draw (1,0.5) node {$\bullet$};
\node [left] at (0,0) {r};
\node [above] at (1,3) {p};
\node [below] at (1,-2) {q};
\node [right] at (1,0.5) {m};

\draw (4,1)--(5,4);
\draw (4,1)--(5,1.5);
\draw (5,4)--(5,1.5);
\node [left] at (4,1) {r};
\node [above] at (5,4) {p};
\node [right] at (5,1.5) {m};
\draw (4.8,1.4)  to [out=100,in=45] (5,1.7);
\draw (5,2.25) node {$\times$};
\draw (4.5,1.25) node {$||$};
\node [above] at (4.8,1.55) {$\gamma$};

\draw (4,-1)--(5,-3);
\draw (4,-1)--(5,-0.5);
\draw (5,-3)--(5,-0.5);
\node [left] at (4,-1) {r};
\node [below] at (5,-3) {q};
\node [right] at (5,-0.5) {m};
\draw (4.7,-0.65)  to [out=-90,in=180] (5,-0.9);
\node [below] at (4.9,-0.77) {$\pi-\gamma$};
\draw (5,-2.25) node {$\times$};
\draw (4.5,-0.75) node {$||$};

\draw (7,1) to [out=90,in=-120]  (8,4);
\draw (7,1)--(8,1.5);
\draw (8,4)--(8,1.5);
\node [left] at (7,1) {r'};
\node [above] at (8,4) {p'};
\node [right] at (8,1.5) {m'};
\draw (7.8,1.4)  to [out=100,in=45] (8,1.7);
\draw (8,2.25) node {$\times$};
\draw (7.5,1.25) node {$||$};
\node [above] at (7.8,1.55) {$\gamma$};

\draw (7,-1) to [out=-90 ,in=125] (8,-3);
\draw (7,-1)--(8,-0.5);
\draw (8,-3)--(8,-0.5);
\node [left] at (7,-1) {r'};
\node [below] at (8,-3) {q'};
\node [right] at (8,-0.5) {m'};
\draw (7.7,-0.65)  to [out=-90,in=180] (8,-0.9);
\node [below] at (7.9,-0.77) {$\pi-\gamma$};
\draw (8,-2.25) node {$\times$};
\draw (7.5,-0.75) node {$||$};
\end{tikzpicture}
\label{fig:comparaison}
\end{figure}

\[
C_\kappa(d'(p',r'))=C_\kappa(d'(m',r'))C_\kappa\left(\frac{d'(p',q')}{2}\right)+S_\kappa(d'(m',r'))S_\kappa\left(\frac{d'(p',q')}{2}\right)\cos(\gamma')
\]
And \[C_\kappa(d'(q',r'))=C_\kappa(d'(m',r;))C_\kappa\left(\frac{d'(p',q')}{2}\right)+S_\kappa(d'(m',r'))S_\kappa\left(\frac{d'(p',q')}{2}\right)\cos(\pi-\gamma')\]
Summing the two equations we get:\[C_\kappa(d'(p',r'))+C_\kappa(d'(q',r'))=2C_\kappa(d'(m',r'))C_\kappa\left(\frac{d'(p',q')}{2}\right)\]

This in turn implies since $\Delta(pqr)$ and $\Delta(p'q'r')$ are comparison triangles that:
$C_\kappa(d(p,r))+C_\kappa(d(q,r))=2C_\kappa(d'(m',r'))C_\kappa\left(\frac{d(p,q)}{2}\right)$
Since $C_\kappa$ is decreasing in $[0,\frac{\pi}{\sqrt{K}}]$ and that $d(r,m) < d(r',m')$, we get
\[C_\kappa(d(p,r))+C_\kappa(d(q,r)) \leq 2C_\kappa(d(m,r))C_\kappa\left(\frac{d(p,q)}{2}\right)\,;\]
Which is the desired result.
\end{proof}

\begin{prop}[\cite{bridson:haefliger:1999}[prop. II.1.4]]
\label{prop:catk}
Let $\mathcal{M}$ denote a $CAT(\kappa)$ metric space.
\begin{enumerate}
\item If $x$ and $y$ in $\mathcal{M}$ are such that $d(x,y)<D_\kappa$, there exists a unique geodesic $[x,y]$ joining them.
\item For any $x\in\mathcal{M}$, the ball $B_{x,r}$ with $r<r_\kappa$ is convex.
\end{enumerate}
\end{prop}

\subsection{Convex Sets}

Convexity can have several meaning in the context of metric spaces (\emph{cf.}, \emph{e.g.} \cite[p.403]{Chavel}).
\begin{defn}[Convexity]
A subset $S$ of $\mathcal{M}$ is said \emph{convex} when for every couple of points $(x,y)\in S^2$, every geodesic segment $\gamma$ joining $x$ and $y$ in $(\mathcal{M},d)$ is such that $\gamma\subset S$.
\end{defn}

The notion of convex hull is going to be useful in the sequel.
\begin{defn}[Convex Hull]
Assume $S$ is a subset of $\mathcal{M}$. Then the \emph{convex hull} of $S$, denoted $\conv(S)$ is the intersection of all closed convex sets containing $S$.
\end{defn}
One can easily check that $\conv(S)$ is indeed convex (and closed).

\begin{prop}
If $(\mathcal{M},d)$ is $CAT(0)$, then for each $x_0\in\mathcal{M}$ and $r\geq 0$, every ball $B_{x_0,r}=\{x\in\mathcal{M}: d(x,x_0)\leq r\}$ is convex.
\end{prop}
\begin{proof}
Consider $x,y\in B_{x_0,r}$ and a comparison triangle $\Delta(x_0,x,y)$. Then for each $z\in[x,y]$, CAT(0) inequality implies $d(x_0,z)\leq \max(d(x_0,x),d(x_0,y))\leq r$. Hence $z\in B_{x_0,r}$, which finishes the proof.
\end{proof}

\bibliographystyle{alpha}
\bibliography{bibliog}
\label{sec:bibliog}

\end{document}